\documentclass[reqno]{amsart}

\usepackage{amsmath}
\usepackage{amsfonts}
\usepackage{amssymb,enumerate}
\usepackage{amsthm}
\usepackage[all]{xy}
\usepackage{rotating}
\usepackage{amscd }
\usepackage{hyperref}

\theoremstyle{plain}
\newtheorem{lem}{Lemma}[section]
\newtheorem{cor}[lem]{Corollary}
\newtheorem{prop}[lem]{Proposition}

\newtheorem{thm}[lem]{Theorem}

\newtheorem{intthm}{Theorem}

\theoremstyle{definition}
\newtheorem{defn}[lem]{Definition}
\newtheorem{ex}[lem]{Example}

\newtheorem{disc}[lem]{Remark}
\newtheorem{construction}[lem]{Construction}

\newtheorem{fact}[lem]{Fact}

\newtheorem{convention}[lem]{Convention}

\newcommand{\cat}[1]{\mathcal{#1}}
\newcommand{\catx}{\cat{X}}

\newcommand{\catm}{\cat{M}}

\newcommand{\catg}{\cat{G}}
\newcommand{\catp}{\cat{P}}
\newcommand{\catf}{\cat{F}}
\newcommand{\cati}{\cat{I}}
\newcommand{\cata}{\cat{A}}

\newcommand{\catb}{\cat{B}}

\newcommand{\catib}{\cat{I}_B}

\newcommand{\caticd}{\cat{I}_{C^{\dagger}}}

\newcommand{\catpcd}{\cat{P}_{C^{\dagger}}}
\newcommand{\catac}{\cat{A}_C}
\newcommand{\catab}{\cat{A}_B}
\newcommand{\catbc}{\cat{B}_C}

\newcommand{\catbb}{\cat{B}_B}
\newcommand{\catacd}{\cat{A}_{\da{C}}}
\newcommand{\catbcd}{\cat{B}_{\da{C}}}

\newcommand{\catic}{\cat{I}_C}

\newcommand{\catpb}{\cat{P}_B}
\newcommand{\catpc}{\cat{P}_C}
\newcommand{\catfc}{\cat{F}'}

\newcommand{\pd}{\operatorname{pd}}

\newcommand{\id}{\operatorname{id}}	
\newcommand{\fd}{\operatorname{fd}}

\newcommand{\catpd}[1]{\cat{#1}\text{-}\pd}
\newcommand{\xpd}{\catpd{X}}

\newcommand{\xid}{\catid{X}}

\newcommand{\catid}[1]{\cat{#1}\text{-}\id}

\newcommand{\pcpd}{\catpc\text{-}\pd}
\newcommand{\fcpd}{\catfc\text{-}\pd}

\newcommand{\icid}{\catic\text{-}\id}

\newcommand{\ann}{\operatorname{Ann}}
\newcommand{\mspec}{\operatorname{m-Spec}}

\newcommand{\len}{\operatorname{len}}

\newcommand{\HH}{\operatorname{H}}
\newcommand{\Hom}{\operatorname{Hom}}	
\newcommand{\coker}{\operatorname{Coker}}
\newcommand{\spec}{\operatorname{Spec}}
\newcommand{\s}{\mathfrak{S}}
\newcommand{\tor}{\operatorname{Tor}}
\newcommand{\im}{\operatorname{Im}}

\newcommand{\da}[1]{#1^{\dagger}}

\newcommand{\Pic}{\operatorname{Pic}}

\newcommand{\End}{\operatorname{End}}

\newcommand{\Ker}{\operatorname{Ker}}

\newcommand{\ideal}[1]{\mathfrak{#1}}
\newcommand{\m}{\ideal{m}}

\newcommand{\p}{\ideal{p}}

\newcommand{\comp}[1]{\widehat{#1}}

\newcommand{\wti}{\widetilde}

\newcommand{\bbz}{\mathbb{Z}}

\newcommand{\from}{\leftarrow}
\newcommand{\xra}{\xrightarrow}

\newcommand{\onto}{\twoheadrightarrow}

\renewcommand{\geq}{\geqslant}
\renewcommand{\leq}{\leqslant}

\renewcommand{\hom}{\Hom}

\newcommand{\Ext}[4][R]{\operatorname{Ext}_{#1}^{#2}(#3,#4)}

\newcommand{\Otimes}[3][R]{#2\otimes_{#1}#3}
\renewcommand{\Hom}[3][R]{\operatorname{Hom}_{#1}(#2,#3)}
\newcommand{\Tor}[4][R]{\operatorname{Tor}^{#1}_{#2}(#3,#4)}

\renewcommand{\catfc}{\cat{F}_C}
\newcommand{\tormpc}[3]{\Tor[\catm\catpc]{#1}{#2}{#3}}
\newcommand{\torpcm}[3]{\Tor[\catpc\catm]{#1}{#2}{#3}}
\newcommand{\tormfc}[3]{\Tor[\catm\catfc]{#1}{#2}{#3}}
\newcommand{\torfcm}[3]{\Tor[\catfc\catm]{#1}{#2}{#3}}
\newcommand{\tormfcm}[3]{\Tor[\catm\catf_{C_{\m}}]{#1}{#2}{#3}}

\newcommand{\torfbmm}[3]{\Tor[\catf_{B_{\m}}\catm]{#1}{#2}{#3}}
\newcommand{\cd}{C^{\dagger}}

\newcommand{\extmic}[3]{\Ext[\catm\catic]{#1}{#2}{#3}}

\newcommand{\catfb}{\cat{F}_B}

\newcommand{\torpbm}[3]{\Tor[\catpb\catm]{#1}{#2}{#3}}

\newcommand{\torfbm}[3]{\Tor[\catfb\catm]{#1}{#2}{#3}}

\newcommand{\extpcm}[3]{\Ext[\catpc\catm]{#1}{#2}{#3}}

\numberwithin{equation}{lem}

\begin{document}

\bibliographystyle{amsplain}

\author[M. Salimi]{Maryam Salimi}
\address{Maryam Salimi\\
Department of Mathematics, Science and Research Branch, Islamic Azad University, Tehran, Iran}

\email{maryamsalimi@ipm.ir}

\author[S. Sather-Wagstaff]{Sean Sather-Wagstaff}

\address{Sean Sather-Wagstaff\\
Department of Mathematics,
NDSU Dept \# 2750,
PO Box 6050,
Fargo, ND 58108-6050
USA}

\email{sean.sather-wagstaff@ndsu.edu}

\urladdr{http://www.ndsu.edu/pubweb/\~{}ssatherw/}

\thanks{This material is based on work supported by North Dakota EPSCoR and
National Science Foundation Grant EPS-0814442. The research of Siamak Yassemi was in part supported by a grant from
IPM (No. 91130214)}

\author[E. Tavasoli]{Elham Tavasoli}
\address{Elham Tavasoli\\
Department of Mathematics, Science and Research Branch, Islamic Azad University, Tehran, Iran}

\email{elhamtavasoli@ipm.ir}

\author[S. Yassemi]{Siamak Yassemi}
\address{Siamak Yassemi\\Department of Mathematics, University
of Tehran, and School of mathematics, Institute for research
in fundamental sciences(IPM), Tehran- Iran}

\email{yassemi@ipm.ir}
\urladdr{http://math.ipm.ac.ir/yassemi/}

\title{Relative Tor functors with respect to a semidualizing module}

\date{\today}



\begin{abstract}
We consider  relative Tor functors built from resolutions described by a semidualizing module
$C$ over a commutative noetherian ring $R$. 
We show that the bifunctors $\torfcm i--$ and $\torpcm i--$, defined using flat-like and
projective-like resolutions, are isomorphic. We show how the vanishing of these functors 
characterizes the finiteness of the homological dimension $\fcpd$, and we use this to 
give a relation between the $\fcpd$ of a given module and that of a pure submodule.
On the other hand, we show that other
relations that one may expect to hold similarly, fail in general.
In fact, such relations force the semidualizing modules under consideration to be trivial.
\end{abstract}

\maketitle


\section*{Introduction}

For the purposes of this paper, relative homological algebra is the study of non-traditional resolutions and 
the (co)homology theories (i.e., relative derived functors) that they define. 
By ``non-traditional'' we mean that these resolutions are not given directly by 
projective, injective, or flat modules, as they are in ``absolute'' homological algebra.
This idea goes back to 
Butler and Horrocks~\cite{butler:cer} and
Eilenberg and Moore~\cite{eilenberg:frha}.
This area has seen a lot of activity recently thanks to
Enochs and Jenda~\cite{enochs:rha}
and 
Avramov and Martsinkovsky~\cite{avramov:aratc}.

Much of the recent work on the derived functors that arise in this context has 
focused on cohomology, i.e., relative Ext; see, e.g., \cite{avramov:aratc,sather:crct,takahashi:hasm}.
The point of this paper is to begin a pointed discussion of the properties of relative Tor.
The relative homology functors that arise in this context come from resolutions that
model projective resolutions and flat resolutions. Specifically, we consider 
proper $\catpc$-resolutions 
and 
proper $\catfc$-resolutions 
where $C$ is a semidualizing module over a commutative noetherian ring $R$. 
(See Section~\ref{sec111231c} for terminology, notation, and foundational results.)

Section~\ref{sec111231a} consists of basic results about these resolutions.
By their nature, these resolutions have some similar properties, but also some different properties;
For instance, Proposition~\ref{lem111221a} shows that proper $\catpc$-resolutions 
behave well with respect to flat ring extensions, but the behavior of proper $\catfc$-resolutions in this context is not clear.
On the other hand, restriction of scalars is well-behaved for proper $\catfc$-resolutions, but not necessarily
for proper $\catpc$-resolutions, as we show in
Proposition~\ref{lem111224b}.

We have four flavors of relative homology in this context.
For instance, given a proper
$\catpc$-resolution $L$ of an $R$-module $M$, we have
$\torpcm iMN
=\HH_i(\Otimes{L}{N})$
for each $R$-module $N$ and each integer $i$.
The module $\tormpc iMN$ is defined using a proper
$\catpc$-resolution of $N$, and similarly, $\torfcm iMN$ and $\tormfc iMN$ are defined using  proper
$\catfc$-resolutions; see Definition~\ref{defn110730a}.

Certain relations between these are obvious.
For instance,  commutativity of tensor product implies that
$\torpcm iMN\cong\tormpc iNM$
and
$\torfcm iMN\cong\tormfc iNM$.
Other relations are not obvious. 
For instance, it is well-known that $\Tor iMN$ can be computed using a projective resolution of $M$
or a flat resolution of $M$. The corresponding result for relative Tor is
our first main theorem, stated next. It is contained in Theorem~\ref{prop110730a'}.

\begin{intthm}\label{intthm120115a}
Let $C$ be a semidualizing $R$-module, and 
let  $M$ and $N$ be $R$-modules.
For each $i$, there  is a natural isomorphism
$\torpcm iMN\cong\torfcm iMN$.
\end{intthm}

This result allows for a certain amount of flexibility for proving results about relative Tor,
as in the absolute case. This is the subject of the rest of
Section~\ref{sec0}.
For instance, when $M$ and $N$ are finitely generated, it is straightforward to show that
$\torpcm iMN$ is finitely generated, while it is not obvious at all that $\torfcm iMN$ is finitely generated.
On the other hand, $\torfcm iMN$ is well-behaved with respect to flat base change, and we get to conclude that 
$\torpcm iMN$ is similarly well-behaved.
See Propositions~\ref{prop111221b} and~\ref{prop111221a}. 
This section concludes with relative versions of
Hom-tensor adjointness, tensor evaluation, and Hom evaluation
in Propositions~\ref{thm111230a}--\ref{thm111230c}.

Given these nice properties, one may be surprised to know that many properties of absolute Tor
do not pass to the relative setting. These differences are the subject of
Section~\ref{sec111223a}. For instance, in
Example~\ref{ex111220a}
we show that in general we have
\begin{gather*}
\torfcm iMN\ncong \tormfc iMN
\\
\torfcm iNM\ncong\torfcm iMN
\\
\torfcm iMN\ncong\Tor iMN.
\end{gather*}
The remainder of this section focuses on two questions.
First, Propositions~\ref{prop111223e}--\ref{prop111229b} and Examples~\ref{ex111229a}--\ref{ex111229x}
provide classes of modules $M,N$ such that the above ``non-isomorsphisms'' are isomorphisms.
Second, starting with Theorem~\ref{prop111223a}, we show that the only way that 
the above ``non-isomorsphisms'' are always isomorphisms is in the trivial case.
For instance, here is Theorem~\ref{prop111223a}.

\begin{intthm}\label{aprop111223a}
Assume that $(R,\m,k)$ is local, and let $B$ and $C$ be semidualizing $R$-modules.
The following conditions are equivalent:
\begin{enumerate}[\rm(i)]
\item \label{aprop111223a1}
$\torfbm iMN\cong\tormfc iMN$ for all $i\geq 0$ and for all $R$-modules $M$, $N$.
\item \label{aprop111223a2}
$\torfbm iB{k}\cong\tormfc iB{k}$ for $i=0$ and some $i> 0$.
\item \label{aprop111223a3}
$\torfbm i{k}C\cong\tormfc i{k}C$ for $i=0$ and some $i> 0$.
\item \label{aprop111223a4}
$B\cong R\cong C$.
\end{enumerate}
\end{intthm}

Section~\ref{sec111231d} discusses $\fcpd$,
the homological dimension obtained from bounded proper $\catfc$-resolutions,  and its relation to relative Tor.
First, in Proposition~\ref{lem111230a}
we note that this is the same homological dimension as the one calculated from bounded acyclic $\catfc$-resolutions.
From this, we deduce some flat base change results for $\fcpd$.
In Theorems~\ref{prop110811d} and~\ref{prop111231a} we prove the next result which characterizes modules of
finite $\fcpd$ in terms of vanishing of relative Tor.

\begin{intthm}\label{intthm120115b}
Let $C$ be a semidualizing $R$-module, and
let $M$  be an $R$-module. Given an integer $n\geq 0$, consider the following conditions:
\begin{enumerate}[\rm(i)]
\item \label{intthm120115b1}
$\torfcm{i} M- = 0$ for all $i>n$;
\item \label{intthm120115b2}
$\torfcm{n+1} M- = 0$; and
\item \label{intthm120115b3}
$\fcpd_R(M)\leq n$.
\item \label{intthm120115b4}
$\torfcm{i} M{R/\m} = 0$ for all $i>n$ and for each $\m\in\mspec(R)$;
\item \label{intthm120115b5}
$\torfcm{n+1} M{R/\m} = 0$ for each $\m\in\mspec(R)$; and
\item \label{intthm120115b6}
$\pcpd_R(M)\leq n$.
\end{enumerate}
The conditions \eqref{intthm120115b1}--\eqref{intthm120115b3} are always equivalent.
If $M$ is finitely generated, then conditions \eqref{intthm120115b1}--\eqref{intthm120115b6} are  equivalent.
\end{intthm}

Section~\ref{sec110818a} contains the following application to pure submodules, motivated by a result of Holm and White~\cite{holm:fear}.
See Theorem~\ref{thm111231a}.

\begin{intthm}\label{xthm111231a}
Let $C$ be a semidualizing $R$-module, and let
$M'\subseteq M$ be a pure submodule. Then one has
$$\fcpd_R(M)\geq\sup\{\fcpd_R(M'),\fcpd_R(M/M')-1\}.$$
\end{intthm}

\section{Background Material}\label{sec111231c}

\begin{convention}
Throughout this paper $R$ and $S$ are commutative noetherian rings, and
$\mathcal{M}(R)$ is the category of $R$-modules.
We use the term ``subcategory of $\catm(R)$" to mean a ``full, additive subcategory
 $\mathcal{X} \subseteq \mathcal{M}(R)$ such that, for all $R$-modules
 $M$ and $N$, if $M\cong N$ and $M \in \mathcal{X}$, then $N \in \mathcal{X}$."
 Write $\catp(R)$, $\catf(R)$ and $\cati(R)$ for the subcategories of
 projective, flat and injective $R$-modules, respectively.
Write $\mspec(R)$ for the set of maximal ideals of $R$.
\end{convention}

\subsection*{General Notions}

\begin{defn}\label{defn111219a}
An \emph{$R$-complex} is a sequence of $R$-module homomorphisms
$$Y = \cdots \stackrel{\partial _{n+1} ^{Y} } \longrightarrow
Y_{n} \stackrel{\partial _{n} ^{Y} } \longrightarrow Y_{n-1}
\stackrel{\partial _{n-1} ^{Y} } \longrightarrow \cdots $$
such that 
$\partial _{n-1}^{Y} \partial _{n} ^{Y} = 0$ for each integer $n$.
When $Y$ is an $R$-complex, set $\HH_n(Y)=\Ker(\partial^Y_n)/\im(\partial^Y_{n+1})$ for each $n$.
Given a subcategory $\catx$ of $\catm(R)$,
an $R$-complex $Y$ is \emph{$\Hom \catx -$-exact} if the complex
$\Hom XY$ is exact for each $X$ in $\catx$.
The term \emph{$\Hom -\catx$-exact} is defined similarly.

Given two $R$-complexes $Y$ and $Z$, a \emph{chain map} $f\colon Y\to Z$ is a 
sequence of $R$-module homomorphisms $\{f_i\colon Y_i\to Z_i\}$ making  the obvious ``ladder-diagram''
commute. A chain map $f\colon Y\to Z$ is a \emph{quasiisomorphism} if the induced map
$\HH_i(f)\colon\HH_i(Y)\to\HH_i(Z)$ is an isomorphism for each $i$.
In general, the complexes $Y$ and $Z$ are \emph{quasiisomorphic} provided that there is a sequence of quasiisomorphisms
$Y\from Y^1\to Y^2\from\cdots\from Y^m\to Z$ for some integer $m$.
\end{defn}

In this paper,  resolutions are built from precovers, and
 coresolutions are built from preenvelopes, defined next.
For more details about precovers and preenvelopes, the reader may consult
\cite [Chapters 5 and 6]{enochs:rha}.

\begin{defn}
Let $\mathcal{X}$ be a subcategory of $\mathcal{M}(R)$ and let $M$ be an $R$-module.
An \emph{$\mathcal{X}$-precover} of $M$ is an $R$-module homomorphism $\varphi \colon X \to M$,
 where $X \in \mathcal{X}$, and such that the sequence
 $$\Hom {X'}{\varphi}  \colon \Hom {X'}{X}  \to
 \Hom {X'}{M} \to 0$$ is exact for every $X' \in \mathcal{X}$.
 If every $R$-module admits $\mathcal{X}$-precover, then the class
 $\mathcal{X}$ is \emph{precovering}. 
 The terms \emph{$\mathcal{X}$-preenvelope} and \emph{preenveloping} are
  defined dually.

Assume that $\mathcal{X}$ is precovering. Then each $R$-module $M$
has an \emph{augmented proper $\mathcal{X}$-resolution}, that is, an $R$-complex
$$X ^ {+} = \cdots \stackrel{\partial _{2} ^{X} } \longrightarrow
X_{1} \stackrel{\partial _{1} ^{X} } \longrightarrow X_{0}
 \xra{\tau} M \longrightarrow 0 $$ such that
$\Hom {Y}{X^{+}}$ is exact for all $Y \in \mathcal{X}$. The truncated complex
$$X  = \cdots \stackrel{\partial _{2} ^{X} } \longrightarrow
X_{1} \stackrel{\partial _{1} ^{X} } \longrightarrow X_{0}
\longrightarrow 0$$ is a \emph{proper $\mathcal{X}$-resolution} of $M$.
The \emph{$\mathcal{X}$-projective dimension} of $M$ is
$$\xpd_R(M)=\inf\{\sup\{n\mid X _{n} \neq 0 \}\mid
\text{$X$  is a proper $\mathcal{X}$-resolution of  $M$}\}.
$$
\emph{Proper
$\mathcal{X}$-coresolutions} and $\xid$ are defined dually. 

When $\mathcal{X}$ is the class of projective $R$-modules,
we write $\pd_{R}( M)$ for the associated homological
dimension and call it the \emph{projective dimension} of $M$.
Similarly, the flat and injective dimensions of $M$ are
denoted $\fd_{R} (M)$ and $\id _{R} (M)$.
\end{defn}

\begin{disc}\label{disc111219a}
Let $\mathcal{X}$ be a precovering subcategory of $\mathcal{M}(R)$.
We note explicitly that augmented proper $\catx$-resolutions need not be exact. 

According to our definitions, we have $\xpd_R(0)=-\infty$.
The modules of $\mathcal{X}$-projective dimension zero are the non-zero
modules in $\mathcal{X}$. 

Note that projective resolutions (in the usual sense) are automatically proper.
Also, note that augmented proper flat resolutions are automatically exact.
\end{disc}

The following result shows that there is some versatility in proper flat resolutions. 
It is for use in Proposition~\ref{lem111230a}.

\begin{lem}\label{lem111231a}
Let $N$ be a module such that there is an exact sequence
$$0\to G_n\to\cdots\to G_0\to N\to 0$$
where each $G_i$ is flat.
Let $F$ be a proper flat resolution of $N$,
and set $K_n=\im(\partial^F_{n+1})$.
Then the truncation
$$\wti F^+=(0\to K_n\to F_{n-1}\xra{\partial^F_{n-1}}\cdots\xra{\partial^F_{1}}F_0\to N\to 0)$$
is also a proper flat resolution of $N$. 
\end{lem}

\begin{proof}
Note that Remark~\ref{disc111219a} implies that $F^+$ is exact, so $\wti F^+$ is also exact.
A standard version of Schanuel's Lemma implies that $K_n$ is flat.
Let $G$ be a flat $R$-module. We need to show that $\Hom G{\wti F^+}$ is exact.
The left exactness of $\Hom G-$ shows that $\Hom G{\wti F^+}$ is exact in degrees $\geq n-1$.
The fact that $F$ is proper provides the exactness in degrees $<n-1$.
\end{proof}

\begin{disc}
\label{disc120116a}
The difference between flat resolutions (in the usual sense) and proper flat resolutions is subtle.
For instance, every $R$-module has a proper flat resolution since $\catf(R)$ is precovering by~\cite{bican:amhfc}.
However, some flat resolutions are proper, and others are not. Moreover, the next example shows that
even bounded flat resolutions need not be proper. On the other hand, Lemma~\ref{lem111231a} shows that
the classical flat dimension of $N$ is the same as $\fd_R(N)$, that is, the homological dimension
defined using flat resolutions (in the usual sense) is the same as 
the homological dimension
defined using proper flat resolutions. See also Proposition~\ref{lem111230a}.
Of course, these subtleties to not come up for $\pd$ and $\id$
since projective resolutions and injective coresolutions are automatically proper.
\end{disc}

\begin{ex}
\label{ex120116a}
Assume that $(R,m,k)$ is a local, non-complete, Gorenstein domain such that $\dim(R)=1$.
For instance, we can take $R=\bbz_{(p)}$ or $k[X]_{(X)}$ where $k$ is a field.
The augmented minimal injective resolution of $R$ (over itself) has the form
$$X=(0\to R\to Q\xra{\alpha} E\to 0)$$
where $Q=Q(R)$ is the field of fractions of $R$ and $E=E_R(k)$ is the injective hull of $k$.
This is also an augmented flat resolution of $E$, in the usual sense. To show that this flat resolution
is not proper, we show that $\Hom{\comp R}{X}$ is not exact 
$$\Hom{\comp R}{X}=(0\to \Hom{\comp R}{R}\to \Hom{\comp R}{Q}\xra{\Hom{\comp R}{\alpha}} \Hom{\comp R}{E}\to 0)$$
where $\comp R$ is the $\m$-adic completion
of $R$. (This suffices since $\comp R$ is flat over $R$.)
In fact, the right-most homology module in this complex is
$$\coker(\Hom{\comp R}{\alpha})=\Ext 1{\comp R}{R}$$
which is non-zero by~\cite[Main Theorem 2.5]{frankild:amsveem}.
See also~\cite[Propositions 4.2 and 4.5]{anderson:neams}
for specific computations of $\Ext 1{\comp R}{R}$.
\end{ex}

\subsection*{Semidualizing Modules and Relative Homological Algebra}

\

Semidualizing modules, defined next, form the basis for our categories of interest.
These objects go back at least to Vasconcelos~\cite{vasconcelos:dtmc}, but were rediscovered by others.

\begin{defn}
A finitely generated $R$-module $C$ is \textit{semidualizing} if the natural
``homothety morphism'' $R \to \Hom CC$ is an isomorphism and $\Ext {i} CC = 0$
for $i\geq 1$.
An $R$-module $D$ is  \emph{dualizing} if it is semidualizing and has finite injective
dimension.

Let $C$ be a semidualizing $R$-module. We set 
\begin{align*}
\catpc(R)
&= \text{the subcategory of modules $M\cong \Otimes{P}{C}$ for some $P\in\catp(R)$}
\\
\catfc(R)
&=\text{the subcategory of modules $M\cong \Otimes{F}{C}$ for some $F\in\catf(R)$}
\\
\catic(R)
&=\text{the subcategory of modules $M\cong \Hom CI$ for some $I\in\cati(R)$.}
\end{align*}
The $R$-modules in $\catpc(R)$, $\catfc(R)$ and $\catic(R)$ are called $C$-\textit{projective},
$C$-\textit{flat} and $C$-\textit{injective}, respectively.
\end{defn}

\begin{disc}\label{disc111219b}
Let $C$ be a semidualizing $R$-module.
In \cite {holm:fear} Holm and White prove that the classes $\mathcal{P}_{C} (R)$ and
$\mathcal{F}_{C} (R)$ are closed under coproducts and summands and the class
$\mathcal{I}_{C} (R)$ is closed under products and summands. Also, they proved
that the classes $\mathcal{P}_{C} (R)$ and $\mathcal{F}_{C} (R)$ are precovering, and 
the class 
$\mathcal{I}_{C}(R)$ is preenveloping.
Since $R$ is noetherian and $C$ is finitely generated,
it is straightforward to show that the class $\catfc(R)$ is closed under products,
 and $\catic(R)$ is closed under coproducts.
\end{disc}

\begin{disc}\label{disc111219bx}
Let $C$ be a semidualizing $R$-module.
Then $C$ 
is cyclic if and only if it is free, if and only if $C\cong R$. Similarly, $\pd_R(C)<\infty$
if and only if $C$ is projective (necessarily of rank 1). 
If $R$ is Gorenstein and local, then $C\cong R$.
If $R\to S$ is a flat ring homomorphism, then $\Otimes SC$ is a semidualizing $S$-module.
In the local setting, these facts are discussed in~\cite[Section 1]{sather:bnsc}.
For the non-local case, see~\cite[Chapter 2]{sather:sdm}.
\end{disc}

The next classes were also introduced by Vasconcelos~\cite{vasconcelos:dtmc}.

\begin{defn}\label{defn111219b}
Let $C$ be a semidualizing $R$-module.
The \emph{Auslander class} with respect to $C$
is the class $\mathcal{A}_{C}(R)$ of $R$-modules $M$ such that:
\begin{itemize}
\item [$(i)$] $\Tor i {C}{M} =0 = \Ext i{C}{\Otimes CM}$ for all $i\geq 1$, and
\item  [$(ii)$] the natural map $M \to \Hom {C}{\Otimes  CM}$
 is an isomorphism.
\end{itemize}
The \emph{Bass class} with respect to $C$ is the class $\mathcal{B}_{C} (R)$ of $R$-modules $M$
such that:
\begin{itemize}
\item [$(i)$]  $\Ext i{C}{M} = 0= \Tor i {C}{\Hom {C}{M}}$ for all $i \geq 1$, and
\item [$(ii)$] the natural evaluation map $\Otimes{C}{\Hom CM}\xra{\xi^C_M} M$
 is an isomorphism.
\end{itemize}
\end{defn}

\begin{disc}\label{disc111220a}
Let $C$ be a semidualizing $R$-module.
The classes $\catac(R)$ and $\catbc(R)$ satisfy the ``two-of-three property'':
given an exact sequence $0\to M_1\to M_2\to M_3\to 0$ of $R$-module homomorphisms,
if two of the $M_i$ are in $\catac(R)$ or in $\catbc(R)$, then so is the third $M_i$;
see~\cite[Corollary 6.3]{holm:fear}.

The class $\catac (R)$ contains all $R$-modules of finite flat dimension and
all modules of finite $\catic$-injective dimension. The Bass class
$\catbc (R)$ contains all $R$-modules of finite injective dimension
and all modules $M$ such that
there is an exact sequence
$$0\to L_n\to\cdots\to L_0\to M\to 0$$ such that each $L_i\in\catfc(R)$;
hence all modules of finite $\catpc$-projective dimension.
(See~\cite[Corollary 6.1]{holm:fear} and~\cite[1.9]{takahashi:hasm}).\footnote{Note that there seems to be a bit of ambiguity 
in~\cite[Corollary 6.1]{holm:fear}.
Before~\cite[1.3]{holm:fear} the authors state that all resolutions are defined by precovers.
In~\cite[1.3]{holm:fear}, the authors define proper resolutions in terms of precovers, but  in~\cite[1.4]{holm:fear}
they define $\xpd$ in terms of $\catx$-resolutions, with no mention of properness.
Then in~\cite[Corollary 6.1]{holm:fear}, the authors are clearly assuming that their bounded augmented resolutions are exact.
For $\pcpd$ and $\icid$, this is covered in~\cite[Corollary 2.10]{takahashi:hasm}, which we recall in Fact~\ref{fact111230a}.
However, $\fcpd$ is not covered there, at least not explicitly. We take care of this in Proposition~\ref{lem111230a}.}
See also Proposition~\ref{lem111230a}.

Foxby equivalence~\cite[Theorem 2.8]{takahashi:hasm} states the following:
\begin{enumerate}[(a)]
\item An $R$-module $M$ is in $\catbc(R)$ if and only if $\Hom CM\in\catac(R)$.
\item An $R$-module $M$ is in $\catac(R)$ if and only if $\Otimes CM\in\catbc(R)$.
\end{enumerate}

The Auslander and Bass classes for $C=R$ are trivial: $\catb_R(R)=\catm(R)=\cata_R(R)$.
\end{disc}

The next two results are for use in the proofs of Propositions~\ref{prop110730b} and~\ref{lem111230a}.

\begin{lem}\label{lem110730a}
Let $C$ be a semidualizing $R$-module,
and let $X$, $Y$ be  $R$-complexes  such that $X_i,Y_i\in\catac(R)$ for each index $i$.
Assume that $X$ and $Y$ are both either bounded above or bounded below.
\begin{enumerate}[\rm(a)]
\item\label{lem110730a1}
If $X$ is exact, then so is $\Otimes C X$.
\item\label{lem110730a2}
If $f\colon X\to Y$ is a quasiisomorphism, then so is $\Otimes Cf\colon \Otimes CX\to \Otimes CY$.
\item\label{lem110730a3}
If $X$ and $Y$ are quasiisomorphic and bounded below, then so are $\Otimes CX$ and $\Otimes CY$.
\end{enumerate}
\end{lem}

\begin{proof}
\eqref{lem110730a1}
The result holds if $X$ is a short exact sequence, since $\Tor 1C{X_i}=0$ for each $i$.
The general result follows by breaking $X$ into short exact sequences.
Note that this uses the two-of-three property for $\catac(R)$
from Remark~\ref{disc111220a}.

\eqref{lem110730a2}
This follows by applying part~\eqref{lem110730a1} to the mapping cone of $f$.

\eqref{lem110730a3}
This follows from part~\eqref{lem110730a2}, since there are quasiisomorphisms
$f\colon P\to X$ and $g\colon P\to Y$
for some bounded below complex $P$ of projective $R$-modules.
Note that this uses the fact that every projective $R$-module is in $\catac(R)$;
see Remark~\ref{disc111220a}.
\end{proof}

Similarly, we have the following.

\begin{lem}\label{lem110730aa}
Let $C$ be a semidualizing $R$-module,
and let $X$, $Y$ be  $R$-complexes  such that $X_i,Y_i\in\catbc(R)$ for each index $i$.
Assume that $X$ and $Y$ are both either bounded above or bounded below.
\begin{enumerate}[\rm(a)]
\item\label{lem110730aa1}
If $X$ is exact, then so is $\Hom C X$.
\item\label{lem110730aa2}
If $f\colon X\to Y$ is a quasiisomorphism, then so is $\Hom Cf\colon \Hom CX\to \Hom CY$.
\item\label{lem110730aa3}
If $X$ and $Y$ are quasiisomorphic and bounded above, then so are $\Hom CX$ and $\Hom CY$.
\end{enumerate}
\end{lem}

Next, we recall some results  from~\cite[Corollary 2.10 and Theorem 2.11]{takahashi:hasm}.
It compares directly to Proposition~\ref{lem111230a}.

\begin{fact}\label{fact111230a}
Let $C$ be a semidualizing $R$-module, and let $M$ be an $R$-module. 
\begin{enumerate}[(a)]
\item \label{fact111230a1}
One has $\pcpd_R(M)\leq n$ if and only if there is an exact sequence
$$0\to L_n\to\cdots\to L_0\to M\to 0$$ such that each $L_i\in\catpc(R)$.
\item \label{fact111230a2}
One has $\icid_R(M)\leq n$ if and only if there is an exact sequence
$$0\to M\to J^0\to\cdots\to J^n\to 0$$
such that each $J^i\in\catic(R)$.
\item \label{fact111230a3}
$\pcpd_R(M)=\pd_R(\Hom CM)$.
\item \label{fact111230a4}
$\icid_R(M)=\id_R(\Otimes CM)$.
\item \label{fact111230a5}
$\pcpd_R(\Otimes CM)=\pd_R(M)$.
\item \label{fact111230a6}
$\icid_R(\Hom CM)=\id_R(M)$.
\end{enumerate}
\end{fact}

The following functors are studied in~\cite{sather:crct,takahashi:hasm}.
We work with them in Propositions~\ref{thm111230a}--\ref{thm111230c},
and use them in the proof of Theorem~\ref{prop110811d}.

\begin{defn}\label{defn111230a}
Let $C$ be a semidualizing $R$-module, and
let $M$ and $N$ be  $R$-modules.
Let $L$ be a proper $\catpc$-resolution of $M$, and let $J$ be a proper $\catic$-coresolution of $N$.
For each $i$, set
\begin{align*}
\extpcm iMN
&:=\HH_{-i}(\Hom LN)
\\
\extmic iMN
&:=\HH_{-i}(\Hom MJ).
\end{align*}
\end{defn}

\begin{fact}\label{fact110811d}
Let $C$ be a semidualizing $R$-module, and
let $M$  be an $R$-module.
Given an integer $n\geq 0$, we know from~\cite[Theorem 3.2(b)]{takahashi:hasm}
that the following conditions are equivalent:
\begin{enumerate}[\rm(i)]
\item \label{fact110811d1}
$\extmic{i} -M = 0$ for all $i>n$;
\item \label{fact110811d2}
$\extmic{n+1} -M = 0$; and
\item \label{fact110811d3}
$\icid_R(M)\leq n$.
\end{enumerate}
\end{fact}

\subsection*{Relations Between Semidualizing Modules}

\

Over a local ring, the ``isomorphism'' relation on the class of semidualizing
modules is pretty good at distinguishing between semidualizing modules
with different properties. For instance, if  $B$ and $C$ are semidualizing modules
over a local ring, then
$\catbc(R)=\catbb(R)$ if and only if $C\cong B$;
see~\cite{frankild:rbsc} for this result and other similar results.
On the other hand, when $R$ is not local, one has to work a bit harder to distinguish
between homologically similar semidualizing modules. 
The following discussion is also from~\cite{frankild:rbsc}.

\begin{defn}\label{defn111224a}
Let $\Pic(R)$ denote the Picard group of $R$.
The elements of $\Pic(R)$ are the isomorphism classes $[P]$ of 
finitely generated rank 1 projective $R$-modules $P$, that is, the finitely
generated projective $R$-modules $P$ such that $P_{\m}\cong R_{\m}$ for
all maximal (equivalently, for all prime) ideals $\m\subset R$.
The group structure on $\Pic(R)$ is given by tensor product
$[P][Q]=[\Otimes PQ]$, and the identity in $\Pic(R)$ is $[R]$. Inverses
are given by duality $[P]^{-1}=[\Hom PR]$, and similarly for division:
$[P]^{-1}[Q]=[\Hom PQ]$.

Let $\s_0(R)$ denote the set of isomorphism classes $[C]$ of semidualizing $R$-modules.
\end{defn}

\begin{fact}\label{fact111224a}
Let $M$ be an $R$-module. Then $M$ is a finitely generated rank 1 projective $R$-module
if and only if $M$ is a semidualizing $R$-module of finite projective dimension,
by~\cite[Remark 4.7]{frankild:rbsc}. So we have $\Pic(R)\subseteq\s_0(R)$.
Also, there is an action of $\Pic(R)$ on $\s_0(R)$ given by $[P][C]=[\Otimes PC]$.
\end{fact}

\begin{defn}\label{defn111224b}
The equivalence relation defined by the action of $\Pic(R)$ on $\s_0(R)$ is denoted $\approx$:
given $[B],[C]\in\s_0(R)$ we have $[B]\approx[C]$ provided that $[B]$ and $[C]$ are in the same
orbit under $\Pic(R)$, that is, provided that there is an element $[P]\in\Pic(R)$ such
that $C\cong\Otimes PB$.
Write $B\approx C$ when $[B]\approx[C]$.
\end{defn}

\begin{fact}\label{fact111224b}
Given semidualizing $R$-modules $B$ and $C$, the following conditions are equivalent:
\begin{enumerate}[(i)]
\item \label{fact111224b1}
$B\approx C$.
\item \label{fact111224b2}
$B_{\m}\cong C_{\m}$ for
all maximal (equivalently, for all prime) ideals $\m\subset R$.
\item \label{fact111224b3}
$\catbb(R)=\catbc(R)$.
\item \label{fact111224b4}
$\catab(R)=\catac(R)$.
\item \label{fact111224b5}
$B\in\catbc(R)$ and $C\in\catbb(R)$.
\end{enumerate}
See~\cite[Theorems 1.4, Propositions 5.1 and 5.4]{frankild:rbsc}.
\end{fact}

\begin{lem}\label{lem111224c}
Let $B$ and $C$ be semidualizing $R$-modules such that $B\approx C$,
and let $[P]\in\Pic(R)$ such that $C\cong \Otimes PB$.
Then one has $\catpb(R)=\catpc(R)$ 
and
$\catfb(R)=\catfc(R)$
and
$\catib(R)=\catic(R)$. 
\end{lem}

\begin{proof}
Let $Q$ be a projective $R$-module. The assumption $C\cong \Otimes PB$
implies that 
$$\Otimes{C}{Q}
\cong\Otimes{(\Otimes PB)}{Q}
\cong\Otimes{B}{(\Otimes PQ)}.
$$
Since $\Otimes PQ$ is projective, this implies that $\catpc(R)\subseteq\catpb(R)$. 
The reverse containment is proved similarly, using the isomorphism
$B\cong\Otimes{\Hom PR}{B}$.
The equalities  $\catfb(R)=\catfc(R)$ and $\catib(R)=\catic(R)$ are verified similarly.
\end{proof}

\subsection*{Two Lemmas on Semidualizing Modules}

\

The next two results are for use in Section~\ref{sec111223a}.

\begin{lem}\label{lem120121a}
Assume that $(R,\m,k)$ is local, and let $C$ be a semidualizing $R$-module.
Consider the following conditions:
\begin{enumerate}[\rm(i)]
\item \label{lem120121a1} $C\cong R$.
\item \label{lem120121a2} $\Otimes CC$ is free.
\item \label{lem120121a3} $\pd_R(\Otimes CC)<\infty$.
\end{enumerate}
Then one has~\eqref{lem120121a1}$\iff$\eqref{lem120121a2}$\implies$\eqref{lem120121a3}.
If $R$ is artinian, then the conditions~\eqref{lem120121a1}--\eqref{lem120121a3} are equivalent.
\end{lem}

\begin{proof}
The implications~\eqref{lem120121a1}$\implies$\eqref{lem120121a2}$\implies$\eqref{lem120121a3} are straightforward. 
When $R$ is artinian, the implication~\eqref{lem120121a3}$\implies$\eqref{lem120121a2}
follows from the Auslander-Buchsbaum formula.

\eqref{lem120121a2}$\implies$\eqref{lem120121a1}
Assume that $\Otimes CC$ is free, and let $\beta=\beta_0(C)$ denote the minimal number of generators of $C$.
By Nakayama's Lemma, the module $\Otimes CC$ is minimally generated by $\beta^2$ many elements, so
we have $\Otimes CC\cong R^{\beta^2}$.
On the other hand, the surjection $R^\beta\onto C$
gives a surjection 
$$C^\beta\onto\Otimes CC\cong R^{\beta^2}$$
by right exactness of tensor product.
This splits, so $R^{\beta^2}$ is a direct summand of $C^{\beta}$. Taking endomorphism rings,
we conclude that $\End(R^{\beta^2})\cong R^{\beta^4}$ is a direct summand of
$\End_R(C^{\beta})\cong R^{\beta^2}$. In particular, this implies that
$\beta^4\leq \beta^2$, which implies that $\beta=1$. It follows that $C$ is cyclic, so $C\cong R$ by
Remark~\ref{disc111219bx}.
\end{proof}

For perspective, the ring $R$ in the next result is isomorphic to the ``trivial extension'' or 
``idealization'' $k\ltimes k^2$.

\begin{lem}\label{lem120121b}
Let $k$ be a field, and set $R=k[X,Y]/(X,Y)^2$. If $C$ is a non-free semidualizing $R$-module, then
$C$ is dualizing for $R$ and $\Otimes CC\cong k^4$.
\end{lem}

\begin{proof}
As $C$ is not free, it is non-cyclic by Remark~\ref{disc111219bx}, so we have $\beta:=\beta_0(C)\geq 2$.
The ring $R$ is local with maximal ideal $\m=(X,Y)R$ such that $\m^2=0$. 

We first show that $C$ is dualizing for $R$.
Since $C$ is non-free, and $\m^2=0$, it follows that there is an exact sequence
$$0\to k^a\to R^{\beta}\to C\to 0$$
with $a\neq 0$. The conditions $\Ext iRC=0=\Ext iCC$ for all $i\geq 1$ imply that
$\Ext i{k^a}C=0$ for all $i\geq 1$. We have $a\neq 0$, so $\Ext i{k}C=0$ for all $i\geq 1$.
Thus $C$ has finite injective dimension, and $C$ is dualizing by definition.

The structure of the dualizing module for this ring is pretty well understood.
For instance, we have $\beta=\mu^0_R(R)=2$. 
Moreover, we can describe $C$ in terms of generators and relations, as follows.
The multi-graded structure on $R$ is represented in the following diagram:
$$
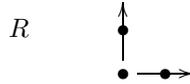
\xymatrix@=.5em{
&&&  &  &  &  &  \\
R&&& \bullet &  &  &  &  \\
&&& \bullet \ar[uu]\ar[rr] & \bullet &}$$
where each bullet represents the corresponding monomial in $R$.
It follows that 
$C\cong E_{R}(k)\cong k\cdot X^{-1}\oplus k\cdot Y^{-1}\oplus k\cdot 1$
with multi-graded module structure given by the formulas
\begin{align*}
X\cdot 1&=0
&X\cdot X^{-1}&=1
&X\cdot Y^{-1}&=0
\\
Y\cdot 1&=0
&Y\cdot Y^{-1}&=1
&Y\cdot X^{-1}&=0.
\end{align*}
In other words, the multi-graded structure is represented by the following diagram:
$$
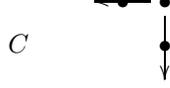
\xymatrix@=.5em{
&&& \bullet & \bullet \ar[ll] \ar[dd] &  \\
C&&& &  \bullet &  \\
&&&&&&&&}
$$
where each bullet represents the corresponding monomial in $C$.
Using this grading, one can show that
$C\supseteq RY^{-1}\cong R/XR$ and
$C/RY^{-1}\cong k$.
In particular, there is an exact sequence
\begin{equation}\label{eq120121a}
0\to R/XR\to C\to k\to 0.
\end{equation}
Also, we see that $XC=k\cdot 1$, so $C/XC\cong k^2$.

We claim that $4\leq\len_R(\Otimes CC)\leq 6$.
To check this, consider the exact sequence
\begin{equation*}
0\to k\to C\to k^2\to 0
\end{equation*}
coming from the equalities $\beta=2$, $\len_R(C)=3$, and $\m^2=0$.
The right exactness of $\Otimes C-$ implies that the next sequence is exact:
$$\Otimes Ck\to \Otimes CC\to \Otimes C{k^2}\to 0.$$
Since $\Otimes Ck\cong k^\beta=k^2$, it follows that
$$4\leq\len_R(\Otimes CC)\leq 4+2=6$$
as claimed.

Next, we show that  $\len_R(\Otimes CC)\leq 4$.
For this, we apply $\Otimes C-$ to the sequence~\eqref{eq120121a}
to obtain the next exact sequence
$$C/XC\to \Otimes CC\to \Otimes Ck\to 0.
$$
As we noted above, we have $C/XC\cong k^2\cong\Otimes Ck$, so additivity of length implies that $\len_R(\Otimes CC)\leq 4$.

It follows that $\len_R(\Otimes CC)= 4$.
Also, we have $\beta_0(\Otimes CC)=\beta_0(C)^2=\beta^2=4$, by Nakayama's Lemma.
That is, the modules $\Otimes CC$ and $(\Otimes CC)/\m(\Otimes CC)$ both have length 4.
Since $(\Otimes CC)/\m(\Otimes CC)$ is a homomorphic image of $\Otimes CC$, it follows that 
$\Otimes CC\cong(\Otimes CC)/\m(\Otimes CC)\cong k^4$
as desired.
\end{proof}

\section{Proper Resolutions}\label{sec111231a}

Throughout this section,  $C$ is a semidualizing $R$-module, and  $M$ is an $R$-module.

\

The results of this section document some properties of proper $\catfc$-resolutions and proper $\catpc$-resolutions.
 We begin with some notation.

\begin{construction}\label{constr111220b}
Let $F$ be a flat (e.g., projective) resolution of $\Hom CM$.
$$F^+=\cdots\xra{\partial^F_2}F_1\xra{\partial^F_1}F_0\xra{\tau}\Hom CM\to 0.$$
Let $\xi^C_M\colon\Otimes{C}{\Hom CM}\to M$ denote the natural evaluation map, and let
$(\Otimes{C}{F})^\pm$ denote the following complex
$$(\Otimes{C}{F})^\pm=\cdots\xra{\Otimes[]{C}{\partial^F_2}}\Otimes C{F_1}
\xra{\Otimes[]{C}{\partial^F_1}}\Otimes C{F_0}\xra{\xi^C_M\circ(\Otimes C\tau)}M\to 0
$$
where $\Otimes{C}{\Hom CM}\xra{\xi^C_M} M$ is the natural evaluation map.
In other words,  $(\Otimes{C}{F})^\pm$ is obtained by augmenting the complex
$\Otimes CF$ by the composition
$$\Otimes C{F_0}\xra{\Otimes C\tau}\Otimes C{\Hom CM}\xra{\xi^C_M}M.
$$
\end{construction}

The next lemma  is  implicit in~\cite{takahashi:hasm}.

\begin{lem}\label{lem111220a}
\begin{enumerate}[\rm(a)]
\item \label{lem111220a1}
If  $F$ is a proper flat resolution of $\Hom CM$, then
$\Otimes CF$ is a proper $\catfc$-resolution of $M$.
\item \label{lem111220a2}
If $G$ is a proper $\catfc$-resolution of $M$, then
$\Hom CG$ is a proper flat resolution of $\Hom CM$.
\item \label{lem111220a3}
If $P$ is a projective resolution of $\Hom CM$, then
$\Otimes CP$ is a proper $\catpc$-resolution of $M$.
\item \label{lem111220a4}
If $Q$ is a proper $\catpc$-resolution of $M$, then
$\Hom CQ$ is a projective resolution of $\Hom CM$.
\end{enumerate}
\end{lem}

\begin{proof}
\eqref{lem111220a1}
To show that $\Otimes CF$ is a proper $\catfc$-resolution of $M$,
it suffices to show that the complex $(\Otimes{C}{F})^\pm$ 
from Construction~\ref{constr111220b} is $\Hom{\catfc}{-}$-exact.
Let $L$ be a flat $R$-module. We need to show that
the complex $\Hom{\Otimes CL}{(\Otimes{C}{F})^\pm}$
is exact. This complex has the following form.
$$
\cdots\!
\to\Hom{\Otimes CL}{\Otimes C{F_1}}
\to\Hom{\Otimes CL}{\Otimes C{F_0}}
\to\Hom{\Otimes CL}{M}\to 0
$$
By Hom-tensor adjointness, this is isomorphic to the next complex where $(-)'=\Hom C{\Otimes C-}$:
$$
\cdots\to
\Hom{L}{F_1'}\to
\Hom{L}{F_0'}\to
\Hom{L}{\Hom{C}{M}}\to
0.$$
Since each $F_i$ is in $\catac(R)$, this is isomorphic to a complex of the following form:
$$
\cdots\to
\Hom{L}{F_1}\to
\Hom{L}{F_0}\to
\Hom{L}{\Hom{C}{M}}\to
0.$$
It is straightforward (but tedious) to show that this complex is isomorphic to
$\Hom{L}{F^+}$ which is exact since $F$ is a proper flat resolution of $\Hom CM$.
Thus, $(\Otimes{C}{F})^\pm$ 
is $\Hom{\catfc}{-}$-exact, as desired.

\eqref{lem111220a2} For each $C$-flat module $Y$, the module
$\Hom{C}{Y}$ is flat. Thus, the fact that $\Hom CG$ is a proper flat resolution of $\Hom CM$
follows as in part~\eqref{lem111220a1}.

Parts~\eqref{lem111220a3} and~\eqref{lem111220a4} are proved similarly.
\end{proof}

\begin{lem}\label{lem111224a}
Let $X$ be an $R$-complex.
Then $X$ is $\Hom{\catpc}{-}$-exact if and only if $\Hom CX$ is exact.
\end{lem}

\begin{proof}
The forward implication is from the condition $C\in\catpc(R)$.
For the reverse implication, let $P$ be a projective $R$-module.
Since $\Hom CX$ is exact,
the fact that $P$ is projective implies that the complex
$\Hom{\Otimes CP}{X}\cong\Hom{P}{\Hom CX}$
is exact, as desired.
\end{proof}

The next result is the first of several  applications of Lemma~\ref{lem111220a}.
We do not know whether the corresponding result for proper $\catfc$-resolutions holds.
See, however, Corollary~\ref{cor111231b}.

\begin{prop}\label{lem111221a}
Let $R\to S$ be a flat ring homomorphism.
If $L$ is a proper $\catpc$-resolution of $M$ over $R$, then
$\Otimes S{L}$ is a proper $\catp_{\Otimes{S}{C}}$-resolution of $\Otimes{S}M$ over $S$.
\end{prop}

\begin{proof}
We augment $\Otimes SL$ in the natural way, via the given augmentation for $L$, so that we have
$(\Otimes SL)^+\cong \Otimes S{(L^+)}$. Thus, the notation $\Otimes SL^+$ is unambiguous.

To show  that $\Otimes S{L}$ is a proper $\catp_{\Otimes{S}{C}}$-resolution of $\Otimes{S}M$,
first note that each module in $L$  is of the form
$L_i\cong\Otimes{C}{P_i}$ for some projective $R$-module $P_i$.
Hence, the module $\Otimes{S}{P_i}$ is projective over $S$.
The isomorphisms
$$\Otimes S{L_i}\cong\Otimes{S}{(\Otimes{C}{P_i})}\cong\Otimes[S]{(\Otimes{S}C)}{(\Otimes{S}{P_i})}$$
imply that $\Otimes{S}{L_i}\in \catp_{\Otimes{S}{C}}(S)$.

Next, since $L$ is a proper $\catpc$-resolution of $M$ over $R$,
the complex $\Hom C{L^+}$ is exact. The fact that $S$ is flat over $R$ implies 
that the next complex
$$\Hom[S]{\Otimes SC}{\Otimes S{L^+}}\cong \Otimes S{\Hom C{L^+}}$$
is also exact. Hence, 
Lemma~\ref{lem111224a} implies that
$\Otimes S{L^+}$ is $\Hom[S]{\catp_{\Otimes SC}}{-}$-exact, 
so $\Otimes SL$ is proper, as desired.
\end{proof}

The previous result works for projectives, but not necessarily for flats.
On the other hand, the next result works for flats, but not for projectives.

\begin{prop}\label{lem111224b}
Let $R\to S$ be a flat ring homomorphism.
Assume that $M$ is an $S$-module, and
let $L$ be a proper $\catf_{\Otimes{S}{C}}$-resolution of $M$ over $S$.
Then $L$ is a proper $\catfc$-resolution of $M$ over $R$.
\end{prop}

\begin{proof}
Each module $L_i$ is of the form 
$L_i\cong\Otimes[S]{(\Otimes{S}{C})}{F_i}\cong\Otimes{C}{F_i}$
for some flat $S$-module $F_i$. Since $S$ is flat over $R$, it follows that
each $F_i$ is flat over $R$, so each $L_i$ is in $\catfc(R)$.
To show that $L$ is proper over $R$, let $G$ be a flat $R$-module:
\begin{align*}
\Hom{\Otimes CG}{L^+}
&\cong\Hom{\Otimes CG}{\Hom[S]{S}{L^+}}\\
&\cong\Hom[S]{\Otimes {S}{(\Otimes CG)}}{L^+} \\
&\cong\Hom[S]{\Otimes[S]{(\Otimes SC)}{(\Otimes SG)}}{L^+}.
\end{align*}
Since  $G$ is flat over $R$, we know that $\Otimes SG$ is flat over $S$, and
it follows that $\Otimes[S]{(\Otimes SC)}{(\Otimes SG)}$ is $\Otimes SC$-flat over $S$.
Thus, the fact that $L$ is proper over $S$ implies that the displayed complexes are exact,
so $L$ is proper over $R$.
\end{proof}

Of course, localization gives useful examples of  flat ring homomorphisms.

\begin{cor}\label{cor111221a}
Let $U$ be a multiplicatively closed subset of $R$.
If $L$ is a proper $\catpc$-resolution of $M$ over $R$, then
$U^{-1}L$ is a proper $\catp_{U^{-1}C}$-resolution of $U^{-1}M$ over $U^{-1}R$.
\end{cor}

\begin{cor}\label{cor111221c}
Let $U$ be a multiplicatively closed subset of $R$, and assume that $M$ is a $U^{-1}R$-module.
If $L$ is a proper $\catf_{U^{-1}C}$-resolution of $M$ over $U^{-1}R$, then
$L$ is a proper $\catfc$-resolution of $M$ over $R$.
\end{cor}

The proofs of parts~\eqref{lem111228a3} and~\eqref{lem111228a2} 
of the next result are necessarily different because $\Hom L-$ does not commute with coproducts in general.

\begin{lem}\label{lem111228a}
Let $\{M_{j}\}_{j\in J}$ be a set of $R$-modules.
For each $j\in J$, let $X_j$ be a proper $\catfc$-resolution of $M_j$, and
let $Y_j$ be a proper $\catpc$-resolution of $M_j$.
\begin{enumerate}[\rm(a)]
\item \label{lem111228a3}
The product $\prod_jX_j$ is a proper $\catfc$-resolution of $\prod_jM_j$.
\item \label{lem111228a2}
The coproduct $\coprod_jY_j$ is a proper $\catpc$-resolution of $\coprod_jM_j$.
\end{enumerate}
\end{lem}

\begin{proof}
\eqref{lem111228a3}
Since $\catfc(R)$ is closed under products by Remark~\ref{disc111219b},
the complex $\prod_jX_j$ consists of modules in $\catfc(R)$.
The augmentation map for $(\prod_jX_j)^+$ is the natural one
induced on products, so we have
$(\prod_jX_j)^+=\prod_jX_j^+$. To show that this complex is $\Hom{\catfc}{-}$-exact,
let $L\in \catfc(R)$ and compute:
$$\textstyle\Hom{L}{\prod_jX_j^+}
\cong\prod_i\Hom{L}{X_j^+}.
$$
Since each complex $\Hom{L}{X_j^+}$ is exact by assumption,
the same is true of the displayed complex, as desired.

\eqref{lem111228a2}
As in part~\eqref{lem111228a3}, the modules in $\coprod_jY_j$ are $C$-projective.
To show that $\coprod_jY_j$ is proper,
Lemma~\ref{lem111224a} shows that we need only check that
$\Hom{C}{\coprod_jY_j^+}$ is exact.
Since $C$ is finitely generated, we know that $\Hom C-$  commutes with coproducts,
so the desired result follows as in the proof of part~\eqref{lem111228a3}.
\end{proof}

\section{Relative Homology} \label{sec0}

In this section,  $C$ is a semidualizing $R$-module, and  $M$ and $N$ are $R$-modules.

\

In our setting, there are four different relative Tor-modules to consider.
They are gotten by resolving in the first slot by modules in $\catpc(R)$ or $\catfc(R)$, and similarly for the second slot.

\begin{defn}\label{defn110730a}
Let $Q$ be a proper $\catpc$-resolution of $M$, and let $G$ be a proper $\catfc$-resolution of $M$.
For each $i\geq 0$, set
\begin{align*}
\torpcm iMN
&:=\HH_i(\Otimes QN)
&
\tormpc iNM
&:=\HH_i(\Otimes NQ)
\\
\torfcm iMN
&:=\HH_i(\Otimes GN)
&
\tormfc iNM
&:=\HH_i(\Otimes NG).
\end{align*}
\end{defn}

\begin{disc}
\label{disc120121a}
The properness assumption on the resolutions in Definition~\ref{defn110730a} guarantee that these
relative Tor constructions  are independent of the choice of resolutions and functorial in both arguments. 
See~\cite[Section 8.2]{enochs:rha}.
Also, there are natural transformations of bifunctors
\begin{align*}
\torpcm 0--
&\to\Otimes --
&
\tormpc 0--
&\to\Otimes --
\\
\torfcm 0--
&\to\Otimes --
&
\tormfc 0--
&\to\Otimes --.
\end{align*}
In general, these are not isomorphisms, as we see in Example~\ref{ex111220a} below. 

Given the symmetric nature of the definitions, one 
has 
\begin{align*}
\torpcm iMN
&\cong
\tormpc iNM
&
\torfcm iMN
&\cong
\tormfc iNM.
\end{align*}
Thus, every result for $\torpcm i--$ has a companion result for $\tormpc i--$,
and similarly for $\torfcm i--$ and $\tormfc i--$.
For the sake of brevity, we do not state both versions explicitly in most cases.
\end{disc}

\newcommand{\catfr}{\cat{F}_R}
\newcommand{\catpr}{\cat{P}_R}
\newcommand{\tormpr}[3]{\Tor[\catm\catpr]{#1}{#2}{#3}}
\newcommand{\torprm}[3]{\Tor[\catpr\catm]{#1}{#2}{#3}}
\newcommand{\tormfr}[3]{\Tor[\catm\catfr]{#1}{#2}{#3}}
\newcommand{\torfrm}[3]{\Tor[\catfr\catm]{#1}{#2}{#3}}

\begin{ex}\label{ex111223a}
In the trivial case $C=R$, 
we have $\catfr(R)=\catf(R)$ and $\catpr(R)=\catp(R)$, and
the relative Tors are the same as the absolute Tors.
$$\torprm i--
\cong\tormpr i--
\cong\torfrm i--
\cong\tormfr i--
\cong\Tor i--$$
\end{ex}

The following long exact sequences come from~\cite[Theorem 8.2.3]{enochs:rha}.

\begin{prop}
\label{prop111220b}
Let $\mathbb{L}=(0\to L'\to L\to L''\to 0)$ be a complex of $R$-modules.
\begin{enumerate}[\rm(a)]
\item \label{prop111220b1}
If $\mathbb{L}$ is $\Hom{\catpc}{-}$-exact (i.e., if $\Hom C{\mathbb L}$ is exact,
e.g., if $L'\in\catbc(R)$),
then there is a long exact sequence
$$\cdots\torpcm 1{L''}{N}\to\torpcm 0{L'}{N}\to\torpcm 0{L}{N}\to\torpcm 0{L''}{N}\to 0$$
that is natural in $\mathbb{L}$ and $N$.
\item \label{prop111220b2}
If $\mathbb{L}$ is $\Hom{\catfc}{-}$-exact, then there is a long exact sequence
$$\cdots\torfcm 1{L''}{N}\to\torfcm 0{L'}{N}\to\torfcm 0{L}{N}\to\torfcm 0{L''}{N}\to 0$$
that is natural in $\mathbb{L}$ and $N$.
\end{enumerate}
\end{prop}

\begin{construction}\label{constr111220a}
For each $i$, there is a natural transformation of bifunctors
$$\varrho_i\colon\torpcm i--\to \torfcm i--.$$
To construct $\varrho_i$,
let  $Q$ be a proper $\catpc$-resolution of $M$,
and let $G$ be a proper $\catfc$-resolution of $M$.
The containment $\catpc(R)\subseteq\catfc(R)$ implies that the augmented resolution
${G}^+$ is $\Hom{\catpc}{-}$-exact.
As in the proof of the functoriality of the relative Tors,
it follows that there is a morphism of complexes
$Q^+\to  G^+$ that is an isomorphism in degree $-1$. 
Furthermore, this morphism is unique up to homotopy.
Thus,  the induced morphism
$\Otimes {Q}N\to\Otimes {G}N$ gives rise to the desired map by taking homology.
\end{construction}

The next result compares to~\cite[Theorem 4.1]{takahashi:hasm} which has similar formulas for relative Ext.
This contains Theorem~\ref{intthm120115a} from the introduction.

\begin{thm}\label{prop110730a'}
For each $i$, there  are natural isomorphisms
$$\torpcm iMN\xra\cong\Tor i{\Hom CM}{\Otimes CN}\xra\cong\torfcm iMN$$
and the morphism $\torpcm i--\xra{\varrho_i}\torfcm i--$ is an isomorphism.
\end{thm}

\begin{proof}
Let $F$ be a proper flat resolution of $\Hom CM$.
Lemma~\ref{lem111220a}\eqref{lem111220a1} implies that $\Otimes CF$ is a proper $\catfc$-resolution of $M$,
so we have
\begin{align*}
\torfcm iMN
&\cong\HH_i(\Otimes{(\Otimes CF)}{N})\\
&\cong\HH_i(\Otimes{F}{(\Otimes CN)})\\
&\cong\Tor i{\Hom CM}{\Otimes CN}.
\end{align*}
The naturality of this isomorphism comes from the naturality of the constructions,
and similarly for $\torpcm iMN$.

Let $P\to F$ be a lift of the identity map on $\Hom CM$.
Then the induced map $(\Otimes{C}{P})^{\pm}\to(\Otimes{C}{F})^{\pm}$ is of the form
$Q^+\to  G^+$, as in Construction~\ref{constr111220a}.
It follows that $\varrho_i(M,N)$ is the map gotten by taking homology
in the map 
$$\Otimes{(\Otimes{C}{P})}N\to\Otimes{(\Otimes{C}{F})}N.$$
Of course, this is equivalent to taking homology
in the map 
$$\Otimes{P}{(\Otimes{C}{N})}\to\Otimes{F}{(\Otimes{C}{N})}.$$
The fact that $\Tor i{\Hom CM}{\Otimes CN}$ can be computed using $P$ or $F$ implies that
the induced maps on homology are isomorphisms, as desired.
\end{proof}

The assumptions on $\mathbb L$ in the next result are satisfied, e.g., when $\mathbb L$ is exact and $L''\in\catac(R)$.

\begin{cor}
\label{cor111220b}
Let $\mathbb{L}=(0\to L'\to L\to L''\to 0)$ be a complex of $R$-modules
such that  $\Otimes C{\mathbb L}$ is exact.
Then there are  long exact sequences
\begin{gather*}
\cdots\torpcm 1{N}{L''}\to\torpcm 0{N}{L'}\to\torpcm 0{N}{L}\to\torpcm 0{N}{L''}\to 0\\
\cdots\torfcm 1{N}{L''}\to\torfcm 0{N}{L'}\to\torfcm 0{N}{L}\to\torfcm 0{N}{L''}\to 0
\end{gather*}
that are natural in $\mathbb{L}$ and $N$.
\end{cor}

\begin{proof}
Apply $\Otimes C-$ to get the exact sequence
$$0\to \Otimes C{L'}\to \Otimes CL\to \Otimes C{L''}\to 0.$$
Now take the long exact sequence in $\Tor i{\Hom CN}{-}$
using Theorem~\ref{prop110730a'}.
\end{proof}

Theorem~\ref{prop110730a'} allows for a certain amount of flexibility for relative Tor,
in the same way that flat and projective resolutions give flexibility for absolute Tor. 
For instance, in the next result, it is not clear that a finitely generated module $M$ has
a proper $\catfc$-resolution $L$ such that each $L_i$ is finitely generated.

\begin{prop}\label{prop111221b}
Assume that $M$ and $N$ are finitely generated over $R$.
For all $i$ the modules $\torpcm iMN$ and $\torfcm iMN$
are finitely generated over $R$.
\end{prop}

\begin{proof}
Since $C$, $M$ and $N$ are finitely generated, so are $\Hom CM$ and $\Otimes CN$,
and hence so is $\Tor i{\Hom CM}{\Otimes CN}$.
Thus, the desired conclusion follows from Theorem~\ref{prop110730a'}.

Alternately, given a degreewise finite $R$-free resolution $F$ of $\Hom CM$,
the complex $\Otimes CF$ is a degreewise finite proper $\catpc$-resolution of $M$
by Lemma~\ref{lem111220a}\eqref{lem111220a3}. It follows that the complex
$\Otimes{(\Otimes CF)}{N}$ is degreewise finite, so the homology modules
$\torpcm iMN\cong\torfcm iMN$
are finitely generated over $R$.
\end{proof}

\begin{prop}\label{prop111228a}
Let $\{N_j\}_{j\in J}$ be a set of $R$-modules.
\begin{enumerate}[\rm(a)]
\item \label{prop111228a1}
For each $i$, there are isomorphisms
\begin{align*}
\textstyle\torpcm i{M}{\coprod_jN_j}
&\textstyle\cong\coprod_j\torpcm i{M}{N_j}
\\
\textstyle\torpcm i{\coprod_jN_j}{M}
&\textstyle\cong\coprod_j\torpcm i{N_j}{M}
\end{align*}
and similarly for $\tor^{\catfc\catm}$.
\item \label{prop111228a2}
If $M$ is finitely generated, then for each $i$, there are isomorphisms
\begin{align*}
\textstyle\torpcm i{M}{\prod_jN_j}
&\textstyle\cong\prod_j\torpcm i{M}{N_j}
\\
\textstyle\torpcm i{\prod_jN_j}{M}
&\textstyle\cong\prod_j\torpcm i{N_j}{M}
\end{align*}
and similarly for $\tor^{\catfc\catm}$.
\end{enumerate}
\end{prop}

\begin{proof}
\eqref{prop111228a1}
For the first isomorphism, let $X$ be a proper $\catpc$-resolution of $M$,
and use the isomorphism
$\Otimes X{\coprod_jN_j}\cong\coprod_j\Otimes X{N_j}$.
For the second isomorphism, use Lemma~\ref{lem111228a}\eqref{lem111228a2}.
The isomorphisms for $\tor^{\catfc\catm}$ follow using Theorem~\ref{prop110730a'}.

\eqref{prop111228a2}
If $M$ is finitely generated, then $\Otimes -M$ commutes with arbitrary products.
Hence, the  isomorphism 
$\torfcm i{\prod_jN_j}{M}
\textstyle\cong\prod_j\torfcm i{N_j}{M}$
follows from Lemma~\ref{lem111228a}\eqref{lem111228a3},
and the corresponding  isomorphisms for $\tor^{\catpc\catm}$ follow using Theorem~\ref{prop110730a'}.
Finally, as in the proof of Proposition~\ref{prop111221b}, the module $M$ has a proper $\catpc$-resolution
$X$ such that each $X_i$ is finitely generated. Hence, the functor $\Otimes X-$ respects arbitrary products,
and the final isomorphisms follow.
\end{proof}

Next, we discuss flat base change. 

\begin{prop}\label{prop111221a}
Let $R\to S$ be a flat ring homomorphism.
Then for all $i$ there are $S$-module isomorphisms
\begin{align*}
\Tor[\catp_{\Otimes{S}{C}}\catm] i{\Otimes{S}{M}}{\Otimes{S}{N}}&\cong\Otimes{S}{\torpcm iMN}
\\
\Tor[\catf_{\Otimes{S}{C}}\catm] i{\Otimes{S}{M}}{\Otimes{S}{N}}&\cong\Otimes{S}{\torfcm iMN}.
\end{align*}
\end{prop}

\begin{proof}
In view of Theorem~\ref{prop110730a'}, it suffices to justify the first isomorphism.
Let $L$ be a 
proper $\catpc$-resolution  of $M$.
Proposition~\ref{lem111221a} implies 
that $\Otimes S{L}$ is a proper $\catp_{\Otimes{S}{C}}$-resolution of $\Otimes{S}M$.
This explains the first step in the next sequence
\begin{align*}
\Tor[\catp_{\Otimes{S}{C}}\catm] i{\Otimes{S}{M}}{\Otimes{S}{N}}
&\cong\HH_i(\Otimes[S]{(\Otimes{S}{L})}{(\Otimes{S}{N})}) \\
&\cong\HH_i(\Otimes{S}{(\Otimes{L}{N})}) \\
&\cong\Otimes{S}{\HH_i(\Otimes{L}{N})} \\
&\cong\Otimes{S}{\torpcm iMN}.
\end{align*}
The third isomorphism is by the $R$-flatness of $S$.
\end{proof}

Of course, localization is a special case of flat base change:

\begin{cor}\label{cor111221b}
Let $U$ be a multiplicatively closed subset of $R$.
Then for all $i$ there are $U^{-1}R$-module isomorphisms
\begin{align*}
\Tor[\catp_{U^{-1}C}\catm] i{U^{-1}M}{U^{-1}N}&\cong U^{-1}\torpcm iMN
\\
\Tor[\catf_{U^{-1}C}\catm] i{U^{-1}M}{U^{-1}N}&\cong U^{-1}\torfcm iMN.
\end{align*}
\end{cor}

\begin{prop}\label{prop111221c}
Let $R\to S$ be a flat ring homomorphism,
and assume that $N$ is an $S$-module.
Then for all $i$ there are $S$-module isomorphisms
\begin{align*}
\Tor[\catp_{\Otimes{S}{C}}\catm] i{\Otimes{S}{M}}{N}&\cong\torpcm iMN
\\
\Tor[\catf_{\Otimes{S}{C}}\catm] i{\Otimes{S}{M}}{N}&\cong\torfcm iMN.
\end{align*}
\end{prop}

\begin{proof}
As in the proof of Proposition~\ref{prop111221a}, the first isomorphism in the next sequence
is from Proposition~\ref{lem111221a}
\begin{align*}
\Tor[\catp_{\Otimes{S}{C}}\catm] i{\Otimes{S}{M}}{N}
&\cong\HH_i(\Otimes[S]{(\Otimes{S}{L})}{N}) 
\cong\HH_i(\Otimes{L}{N}) 
\cong\torpcm iMN.
\end{align*}
This is the first of our desired isomorphisms; the second one follows by~\ref{prop110730a'}.
\end{proof}

\begin{prop}\label{prop111221d}
Let $R\to S$ be a flat ring homomorphism,
and assume that $M$ is an $S$-module.
Then for all $i$ there are $S$-module isomorphisms
\begin{align*}
\Tor[\catp_{\Otimes{S}{C}}\catm] i{\Otimes{S}{M}}{N}&\cong\torpcm iMN
\\
\Tor[\catf_{\Otimes{S}{C}}\catm] i{\Otimes{S}{M}}{N}&\cong\torfcm iMN.
\end{align*}
\end{prop}

\begin{proof}
Let $L$ be a proper $\catf_{\Otimes{S}{C}}$-resolution of $M$ over $S$.
Proposition~\ref{lem111224b}  shows that $L$ is a proper $\catfc$-resolution of $M$ over $R$.
The desired isomorphisms 
now follow as in the proof of Proposition~\ref{prop111221c}.
\end{proof}

The next three results provide relative versions of some standard results for absolute homology,
beginning with Hom-tensor adjointness.

\begin{prop} \label{thm111230a}
Let $I$ be an injective $R$-module.
For all $i \geq 0$ one has
\begin{align}
\label{thm111230a1} \extpcm {i} {M} {\hom NI} &\cong \hom {\torpcm {i} MN}{I}\\
\label{thm111230a2} \extmic {i} {M} {\hom NI} &\cong \hom {\tormpc {i} MN}{I}.
\end{align}
\end{prop}

\begin{proof}
The first isomorphism in the next sequence follows from~\cite[Theorem 4.1]{takahashi:hasm}
\begin{align*}
\extpcm {i} {M} {\hom NI}& \cong \Ext {i}{\Hom CM} {\Hom {C}{\hom NI}}\\
&\cong \Ext {i}{\Hom CM} {\hom {C \otimes _{R} N}{I}}\\
&\cong \hom {\Tor {i} {\Hom CM}{C \otimes _{R} N}}{I}\\
&\cong \hom {\torpcm {i} {M}{N}}{I}.
\end{align*}
The second isomorphism is by Hom-tensor adjointness, and the remaining steps follow 
from~\cite[Theorem 3.2.1]{enochs:rha} and
Theorem~\ref{prop110730a'}.
This explains~\eqref{thm111230a1}, and~\eqref{thm111230a2} is established similarly.
\end{proof}

The next result is a version of tensor evaluation for relative Ext.

\begin{prop}\label{thm111230b}
Assume that  $M$ is finitely generated over $R$, and  let $F$ be a flat $R$-module.
For all $i \geq 0$ there are isomorphisms
\begin{align}
\label{thm111230b1}\extmic {i} MN \otimes _{R} F &\cong \extmic {i} {M} {N \otimes _{R} F}\\
\label{thm111230b2}\extpcm {i} MN \otimes _{R} F& \cong \extpcm {i} {M}{N \otimes _{R} F}.
\end{align}
\end{prop}

\begin{proof}
The  isomorphism~\eqref{thm111230b1} follows from the next display
\begin{align*}
\extmic {i} MN \otimes _{R} F & \cong \Ext {i}{C \otimes _{R} M} {C \otimes _{R} N} \otimes _{R} F\\
&\cong \Ext {i}{C \otimes _{R}M} {(C \otimes _{R} N) \otimes _{R} F}\\
&\cong \Ext {i}{C \otimes _{R}M} {C \otimes _{R} (N\otimes _{R} F)}\\
&\cong \extmic {i} {M} {N \otimes _{R} F}
\end{align*}
which is from~\cite[Theorem 4.1]{takahashi:hasm} and~\cite[Theorem 3.2.15]{enochs:rha}.
The isomorphism~\eqref{thm111230b2} is established similarly.
\end{proof}

Next, we have a version of Hom-evaluation for the relative setting.

\begin{prop}\label{thm111230c}
Assume that $M$ is finitely generated over $R$, 
and
let $I$ be an injective $R$-module.
Then for all $i \geq 0$ there are isomorphisms
\begin{align}
\label{thm111230c1} \torpcm {i} {M} {\hom NI} &\cong \hom {\extpcm {i} MN}{I}\\
\label{thm111230c2} \tormpc {i} {M}{\hom NI}  &\cong \hom {\extmic {i} MN}{I}.
\end{align}
\end{prop}

\begin{proof}
The first isomorphism in the next display is from Theorem~\ref{prop110730a'}:
\begin{align*}
\torpcm {i}{M} {\hom {N}{I}}& \cong \Tor {i}{\Hom CM} {C \otimes _{R} \hom N I}\\
&\cong \Tor {i}{\Hom CM} { \hom {\Hom CN}{ I}}\\
&\cong \hom  {\Ext {i} {\Hom CM} {\Hom C N}} {I}\\
&\cong \hom  {\extpcm {i}{M}{N}} {I}.
\end{align*}
The second and third isomorphisms are from~\cite[Theorem 3.2.11 and 3.2.13]{enochs:rha}, 
and the fourth isomorphism follows from \cite [Theorem 4.1]{takahashi:hasm}.
This explains~\eqref{thm111230c1}, and~\eqref{thm111230c2} is established similarly.
\end{proof}

\section{Comparison of Relative Homologies}\label{sec111223a}

In this section,  $B$, $C$  are semidualizing $R$-modules, and  $M$, $N$ are $R$-modules.

\

Using Theorem~\ref{prop110730a'}, we show that relative Tors do not satisfy the naive version of balance, that they are not commutative,
and that they do not agree with absolute Tor in general.

\begin{ex}\label{ex111220a}
Assume that $(R,\m,k)$ is local and that $C$  is not free, that is,
 that $C$ is not cyclic. 
We show that 
\begin{gather*}
\torfcm ikC\ncong \tormfc ikC
\\
\torfcm iCk\ncong\torfcm ikC
\\
\torfcm ikC\ncong\Tor ikC
\end{gather*}
for all $i$, at least in a specific example.

Let $\beta=\beta_0^R(C)\geq 2$.
It is straightforward to show that
$\tormfc ikC=0$ for all $i\geq 1$ and that $\tormfc 0kC\cong \Otimes kC\cong k^\beta$;
see also Proposition~\ref{prop111223e} and Theorem~\ref{prop110811d}.
From Theorem~\ref{prop110730a'}, we have
$$\torfcm 0kC\cong \Otimes{\Hom Ck}{(\Otimes CC)}\cong\Otimes{k^\beta}{(\Otimes CC)}\cong k^{\beta^3}.$$
This is not isomorphic to 
$$\Tor 0Ck\cong\torfcm 0Ck\cong k^\beta\cong \tormfc 0kC$$ 
as $\beta\geq 2$, so 
$\torfcm 0Ck\ncong\torfcm 0kC\ncong \tormfc 0kC$
and $\torfcm 0kC\ncong\Tor 0kC$.

Again using Theorem~\ref{prop110730a'}, for $i\geq 1$ we have
\begin{equation}\label{eq111223a}
\torfcm ikC\cong \Tor i{\Hom Ck}{\Otimes CC}\cong\Tor ik{\Otimes CC}^{\beta}
\end{equation}
and
$$\torfcm iCk=0= \tormfc ikC.$$
Thus, to show that 
$\torfcm iCk\ncong\torfcm ikC\ncong \tormfc ikC$ in general, it suffices to find an example such that 
$\Tor ik{\Otimes CC}\neq 0$ for all $i\geq 1$, that is, such that
$\pd_R(\Otimes CC)=\infty$.\footnote{We believe that this is true in general, under the assumption that
$C$ is not free; see~\cite{frankild:rbsc}.} This is supplied by Lemma~\ref{lem120121a}, assuming that $R$ is artinian.

Finally, we give a specific  example  where
$\torfcm ikC\ncong\Tor ikC$
for all $i$.
Note that~\eqref{eq111223a} shows that
$\torfcm ikC\cong k^{\beta\cdot\beta_i(\Otimes CC)}$.
Since $\Tor ikC\cong k^{\beta_i(C)}$, it suffices to provide an example
where 
\begin{equation}\label{eq111223b}
\beta\cdot\beta_i(\Otimes CC)>\beta_i(C)
\end{equation}
for all $i\geq 1$.

Set $R=k[X,Y]/(X,Y)^2$, so we have $\m^2=0$. 
Let $C=\Hom[k]Rk$, which is dualizing for $R$ and has $\beta=\mu^0_R(R)=2$. 
Lemma~\ref{lem120121b} implies that $\Otimes CC\cong k^4$, 
so we have 
\begin{equation}\label{eq120121b}
\beta_i(\Otimes CC)=4\beta_i(k)=4\cdot 2^i=2^{i+2} 
\end{equation}
for all $i\geq 0$.
Also, we have $\len_R(C)=3$. Since $\m^2=0$, it follows that there is an exact sequence
\begin{equation}\label{eq111223d}
0\to k^3\to R^2\to C\to 0
\end{equation}
which implies that $\beta_1(C)=3$.
Dimension shifting in the sequence~\eqref{eq111223d} implies that
$$\beta_i(C)=3\beta_{i-1}(k)=3\cdot 2^{i-1}$$ 
for all $i\geq 1$. 
From this, one easily deduces the inequality~\eqref{eq111223b} for all $i\geq 1$,
using~\eqref{eq120121b} with the equality $\beta=2$.
\end{ex}

In general, Example~\ref{ex111220a} shows that many naive properties fail for relative homology.
We continue this section by giving some special cases where these naive properties do hold.

\begin{prop}\label{prop111223e}
If the natural map $\Otimes  {C}{\Hom  CM}\to M$
 is an isomorphism
(e.g., if $M\in\catbc(R)$), then 
$\torfcm 0M-\cong\Otimes M-$.
\end{prop}

\begin{proof}
Again, by Theorem~\ref{prop110730a'} we have
\begin{align*}
\torfcm 0MN
&\cong\Otimes{\Hom CM}{(\Otimes CN)}\\
&\cong\Otimes{(\Otimes{\Hom CM}{C})}{N}\\
&\cong\Otimes MN
\end{align*}
where the last isomorphism is from the assumption $\Otimes  {C}{\Hom  CM}\cong M$.
\end{proof}

In general, we have $\torfcm iM-\ncong\Tor iM-$ by Example~\ref{ex111220a},
even when $M\in\catbc(R)$.
The next result gives  conditions on $M$ and $N$
guaranteeing that the isomorphism $\torfcm iMN\cong\Tor iMN$ does hold.

\begin{prop}\label{prop110730b}
If $M\in\catbc(R)$ and $N\in\catac(R)$, then
for each $i$ there   are isomorphisms
\begin{align*}
\torpcm iMN
&\cong\Tor i{M}{N}\\
\torfcm iMN&\cong\Tor i{M}{N}.
\end{align*}
\end{prop}

\begin{proof}
Let ${P}$ be a  projective resolution of $\Hom CM$,
and let ${Q}$ be a projective resolution of $N$.
Lemma~\ref{lem111220a}\eqref{lem111220a3} implies that  $\Otimes {C}{{P}}$ is
a proper $\catpc$-resolution of~$M$.%

We use the tensor product of complexes.
Since ${Q}$ is a bounded below complex of projective $R$-modules, it respects quasiisomorphisms.
This explains the second isomorphism in the next sequence:
\begin{align*}
\Tor{i}{M}{N}
&\cong\HH_i(\Otimes{M}{{Q}})\\
&\cong\HH_i(\Otimes{(\Otimes {C} {{P}})}{{Q}})\\
&\cong\HH_i(\Otimes{(\Otimes {C} {{P}})}{N})\\
&\cong\torpcm iMN.
\end{align*}
The first isomorphism is from the balance of Tor.
The fourth isomorphism is by definition.
It remains to explain the third isomorphism.

Since ${Q}$ is a projective resolution of $N$, there is a quasiisomorphism ${Q}\xra\simeq N$.
Since ${P}$ is a bounded below complex of projective $R$-modules, the functor
$\Otimes P-$ respects quasiisomorphisms.
So there is a quasiisomorphism $\Otimes {{P}}{{Q}}\xra\simeq \Otimes {{P}}{N}$.
Lemma~\ref{lem110730a}\eqref{lem110730a2} implies that the induced map
$\Otimes C{\Otimes {{P}}{{Q}}}\xra\simeq \Otimes C{\Otimes {{P}}{N}}$ is also a quasiisomorphism.
By the associativity of tensor product, this implies that the complexes
$\Otimes{(\Otimes {C}{{P}})}{{Q}}$ and $\Otimes{(\Otimes {C}{{P}})}{N}$ are quasiisomorphic;
in particular, they have isomorphic homologies, as desired.
\end{proof}

The best results (as best we know) for balance and commutativity are the following.

\begin{prop}\label{prop111229a}
If $M\in\catbb(R)\cap\catac(R)$ and $N\in\catbc(R)\cap\catab(R)$, then one has
$\torfbm iMN\cong\tormfc iMN$ for all $i\geq 0$.
\end{prop}

\begin{proof}
Proposition~\ref{prop110730b} implies that
$$\torfbm iMN\cong\Tor iMN\cong\tormfc iMN$$
for all $i\geq 0$.
\end{proof}

The next result is proved similarly.

\begin{prop}\label{prop111229b}
If $M,N\in\catbc(R)\cap\catac(R)$, then one has
$\torfcm iMN\cong\torfcm iNM$ for all $i\geq 0$.
\end{prop}

We next give examples of some modules that satisfy the hypotheses of the previous two results.
First, we show how to find some modules in $\catbc(R)\cap\catab(R)$ and $\catbb(R)\cap\catac(R)$.

\begin{ex}\label{ex111229a}
By~\cite[Corollary 3.8]{frankild:rbsc}, the following conditions are equivalent:
\begin{enumerate}[(i)]
\item 
$C\in\catab(R)$.
\item
$B\in\catac(R)$, and
\item
$\Tor iBC=0$ for all $i\geq 1$ and $\Otimes BC$ is a semidualizing $R$-module.
\end{enumerate}
For instance, if $A$ is a semiduaizling $R$-module such that 
$A\in\catbc(R)$, then $B=\Hom CA$ satisfies these conditions with $\Otimes BC\cong A$.

Assume that the above conditions  are satisfied. 
Then $B\in\catbb(R)\cap\catac(R)$, and it follows that $\catfb(R)\subseteq\catbb(R)\cap\catac(R)$.
By the two-of-three property from
Remark~\ref{disc111220a}, every module of finite $\catfb$-projective dimension is in $\catbb(R)\cap\catac(R)$.
Similarly, every module of finite $\catfc$-projective dimension is in $\catbc(R)\cap\catab(R)$.

Another class of modules like this is from~\cite[Fact 3.13]{sather:crct}.\footnote{Since 
this example is only given to put our results in perspective, we refer the reader to~\cite{sather:crct} for
the relevant notations and definitions.}
Assume that $R$ is Cohen-Macaulay with a dualizing module $D$.
Then $D\in\catbc(R)$, so the dual $\cd:=\Hom CD$ is a semidualizing $R$-module such that
$C\in\catacd(R)$.
Every module of finite $\catg(\catpc)$-projective dimension is in $\catbc(R)\cap\catacd(R)$,
and every module of finite $\catg(\catpcd)$-projective dimension is in $\catbcd(R)\cap\catac(R)$
by symmetry since $C\cong C^{\dagger\dagger}$.
Also, every module of finite $\catg(\caticd)$-injective dimension is in $\catbc(R)\cap\catacd(R)$,
and every module of finite $\catg(\catic)$-injective dimension is in $\catbcd(R)\cap\catac(R)$.
\end{ex}

Finding modules that are in $\catac(R)\cap\catbc(R)$ is more difficult in general.

\begin{ex}\label{ex111229x}
Assume that $R$ is a domain. Then the quotient field $Q(R)$ is both flat and injective, so it is
in $\catac(R)\cap\catbc(R)$ for each semidualizing $R$-module $C$.
\end{ex}

Of course, if $B\cong R\cong C$, then we have
$\torfbm iMN\cong\Tor iMN\cong\tormfc iMN$; see Example~\ref{ex111223a}.
The following result shows that, in the local case, this is the only way to achieve balance
of all $M$ and $N$; it is Theorem~\ref{aprop111223a} from the introduction. 
We discuss the non-local case below because it requires
more technology.

\begin{thm}\label{prop111223a}
Assume that $(R,\m,k)$ is local.
The following  are equivalent:
\begin{enumerate}[\rm(i)]
\item \label{prop111223a1}
$\torfbm iXY\cong\tormfc iXY$ for all $i\geq 0$ and for all $R$-modules $X$, $Y$.
\item \label{prop111223a2}
$\torfbm iB{k}\cong\tormfc iB{k}$ for $i=0$ and some $i\geq 1$.
\item \label{prop111223a3}
$\torfbm i{k}C\cong\tormfc i{k}C$ for $i=0$ and some $i\geq 1$.
\item \label{prop111223a4}
$B\cong R\cong C$.
\end{enumerate}
\end{thm}

\begin{proof}
We verify the implications~\eqref{prop111223a1}$\implies$\eqref{prop111223a2}$\implies$\eqref{prop111223a4}$\implies$\eqref{prop111223a1}.
The implications~\eqref{prop111223a1}$\implies$\eqref{prop111223a3}$\implies$\eqref{prop111223a4}$\implies$\eqref{prop111223a1} are verified similarly.
Of course, the implication~\eqref{prop111223a1}$\implies$\eqref{prop111223a2}
is trivial, and the implication~\eqref{prop111223a4}$\implies$\eqref{prop111223a1} 
is from Example~\ref{ex111223a}.

\eqref{prop111223a2}$\implies$\eqref{prop111223a4}
We exploit Theorem~\ref{prop110730a'}:
\begin{align*}
\torfbm iBk&\cong \Tor i{\Hom BB}{\Otimes Bk}\\
&\cong\Tor iR{k^{\beta_0(B)}}
\\
&\cong\begin{cases}
k^{\beta_0(B)} &\text{if $i=0$} \\ 0 & \text{if $i\neq 0$}
\end{cases}
\\
\tormfc iBk
&\cong \Tor i{\Otimes CB}{\Hom Ck}\\
&\cong \Tor i{\Otimes CB}{k^{\beta_0(C)}}
\\
&\cong k^{\beta_i(\Otimes CB)\beta_0(C)}.
\end{align*}
Assuming that 
$\torfbm 0B{k}\cong\tormfc 0B{k}$, we conclude that
$$\beta_0(B)=\beta_0(\Otimes CB)\beta_0(C)=\beta_0(B)\beta_0(C)^2.$$
Since $\beta_0(B)\neq 0$, it follows that $\beta_0(C)=1$.
So $C$ is cyclic, and therefore $C\cong R$
by Remark~\ref{disc111219bx}.
Assuming 
$\torfbm iB{k}\cong\tormfc iB{k}$
for some $i\geq 1$, we have
$$0=\beta_i(\Otimes CB)\beta_0(C)=\beta_i(B).$$
It follows that $\pd_R(B)<\infty$, so Remark~\ref{disc111219bx} implies that $B\cong R$, as desired.
\end{proof}

Here is a similar result for commutativity.

\begin{cor}\label{prop111223c}
Assume that $(R,\m,k)$ is local.
The following  are equivalent:
\begin{enumerate}[\rm(i)]
\item \label{prop111223c1}
$\torfcm iXY\cong\torfcm iYX$ for all $i\geq 0$ and for all $R$-modules $X$, $Y$.
\item \label{prop111223c1'}
$\tormfc iXY\cong\tormfc iYX$ for all $i\geq 0$ and for all $R$-modules $X$, $Y$.
\item \label{prop111223c2}
$\torfcm iC{k}\cong\torfcm ikC$ for $i=0$ and some $i\geq 1$.
\item \label{prop111223c2'}
$\tormfc iC{k}\cong\tormfc ikC$ for $i=0$ and some $i\geq 1$.
\item \label{prop111223c3}
$C\cong R$.
\end{enumerate}
\end{cor}

\begin{proof}
This follows from Theorem~\ref{prop111223a} with $B=C$.
For instance, we verify the implication~\eqref{prop111223c2}$\implies$\eqref{prop111223c3}.
Assume that $\torfcm iC{k}\cong\torfcm ikC$ for $i=0$ and some $i\geq 1$. Then
Remark~\ref{disc120121a} implies that
$$\torfcm iC{k}\cong\torfcm ikC\cong\tormfc iCk$$
for $i=0$ and some $i\geq 1$, so we have $C\cong R$ by 
Theorem~\ref{prop111223a}\eqref{prop111223a2}$\implies$\eqref{prop111223a4}.
\end{proof}

Here is another result of the same flavor.

\begin{cor}\label{prop111223d}
Assume that $(R,\m,k)$ is local.
The following  are equivalent:
\begin{enumerate}[\rm(i)]
\item \label{prop111223d1}
$\torfcm iXY\cong\Tor iXY$ for all $i\geq 0$ and for all $R$-modules $X$, $Y$.
\item \label{prop111223d1'}
$\tormfc iXY\cong\Tor iXY$ for all $i\geq 0$ and for all $R$-modules $X$, $Y$.
\item \label{prop111223d2}
$\torfcm iC{k}\cong\Tor iC{k}$ for some $i\geq 1$.
\item \label{prop111223d2'}
$\tormfc i{k}C\cong\Tor i{k}C$ for some $i\geq 1$.
\item \label{prop111223d3}
$C\cong R$.
\item \label{prop111223d4}
$\torfcm ik{k}\cong\Tor ik{k}$ for some $i\geq 0$.
\item \label{prop111223d4'}
$\tormfc i{k}k\cong\Tor i{k}k$ for some $i\geq 0$.
\end{enumerate}
\end{cor}

\begin{proof}
\eqref{prop111223d2'}$\implies$\eqref{prop111223d3}
Example~\ref{ex111223a}
implies that $\Tor i--
\cong\torfrm i--$.
Thus, we can apply Theorem~\ref{prop111223a} with $B=R$.
Indeed, Proposition~\ref{prop111223e} implies that
$\tormfc 0{k}C\cong\Tor 0{k}C\cong\tormfr 0kC$.
Condition~\eqref{prop111223d2'} translates to say that
$\tormfc i{k}C\cong\Tor i{k}C\cong\tormfr ikC$
for some $i\geq 1$, so we have $C\cong R$ by 
Theorem~\ref{prop111223a}\eqref{prop111223a3}$\implies$\eqref{prop111223a4}.

The equivalence of conditions~\eqref{prop111223d1}--\eqref{prop111223d3}
follows similarly from Theorem~\ref{prop111223a} with $B=R$.
And Example~\ref{ex111223a} justifies the implications 
\eqref{prop111223d3}$\implies$\eqref{prop111223d4}
and
\eqref{prop111223d3}$\implies$\eqref{prop111223d4'}.

\eqref{prop111223d4}$\implies$\eqref{prop111223d3}
Since $\Hom Ck\cong k^{\beta_0(C)}\cong\Otimes Ck$,
we have
\begin{gather*}
\torfcm ik{k}\cong\Tor i{\Hom Ck}{\Otimes Ck}
\cong\Tor ikk^{\beta_0(C)^2} 
\cong k^{\beta_0(C)^2\beta_i(k)}
\end{gather*}
and of course
$\Tor ik{k}
\cong k^{\beta_i(k)}$.
Suppose that $\torfcm ik{k}\cong\Tor ik{k}$ for some $i\geq 0$.
It follows that $\beta_i(k)=\beta_0(C)^2\beta_i(k)$,
so either $\beta_0(C)=1$ or $\beta_i(k)=0$.
In the first case, we have $C\cong R$ as before.
In the other case, the ring $R$ is regular, hence Gorenstein, so 
$C\cong R$ by Remark~\ref{disc111219bx}.

The implication \eqref{prop111223d4'}$\implies$\eqref{prop111223d3}
is verified similarly.
\end{proof}

\begin{disc}
Note that Theorem~\ref{prop111223a} and Corollary~\ref{prop111223c}
do not contain versions of 
the conditions~\eqref{prop111223d4} and~\eqref{prop111223d4'} of Corollary~\ref{prop111223d}.
Indeed, for Corollary~\ref{prop111223c} this is because we always have
$$\torfcm ik{k}\cong k^{\beta_0(C)^2\beta_i(k)}\cong\tormfc ikk$$
as the proof of Corollary~\ref{prop111223c} shows.
Similarly, if one assumes in Theorem~\ref{prop111223a} 
that $\torfbm i{k}k\cong\tormfc i{k}k$, then the only conclusion one would be able to draw from this is that
$\beta_0(B)=\beta_0(C)$, which is not enough to guarantee that $B$ and $C$ are isomorphic,
let alone isomorphic to $R$.
\end{disc}

Now we prove the non-local versions of the  results~\ref{prop111223a}--\ref{prop111223d}.
First, we have the following.
Recall the relation $\approx$ from Definition~\ref{defn111224b}.

\begin{prop}\label{prop111224a}
Assume that $B\approx C$,
and let $[P]\in\Pic(R)$ such that $C\cong \Otimes PB$.
For each $i$, there are natural isomorphisms
\begin{align*}
\torfcm i{M}{N}&\cong\torfbm i{M}N\\
\torpcm i{M}{N}&\cong\torpbm i{M}N
\end{align*}
\end{prop}

\begin{proof}
This follows immediately from Lemma~\ref{lem111224c}.
\end{proof}

\begin{cor}\label{cor111225a}
Let $[C]\in\Pic(R)$.
For each $i$ there   are isomorphisms
\begin{align*}
\torpcm iMN
&\cong\Tor i{M}{N}\\
\torfcm iMN&\cong\Tor i{M}{N}.
\end{align*}
\end{cor}

\begin{proof}
The condition $[C]\in\Pic(R)$ is equivalent to $C\approx R$. So, 
the result follows from Example~\ref{ex111223a} and Proposition~\ref{prop111224a}.
\end{proof}

\begin{cor}\label{prop111223b}
The following conditions are equivalent:
\begin{enumerate}[\rm(i)]
\item \label{prop111223b1}
$\torfbm iXY\cong\tormfc iXY$ for all $i\geq 0$ and for all $R$-modules $X$, $Y$.
\item \label{prop111223b2}
$\torfbm iB{R/\m}\cong\tormfc iB{R/\m}$ for  $i=0$, for some $i\geq 1$, and for all $\m\in\mspec(R)$.
\item \label{prop111223b3}
$\torfbm i{R/\m}C\cong\tormfc i{R/\m}C$ for  $i=0$, for some $i\geq 1$, and for all $\m\in\mspec(R)$.
\item \label{prop111223b4}
$B\approx R\approx C$, i.e., $[B],[C]\in\Pic(R)$.
\end{enumerate}
\end{cor}

\begin{proof}
As in the proof of Theorem~\ref{prop111223a}, we verify the implications~\eqref{prop111223b2}$\implies$\eqref{prop111223b4}$\implies$\eqref{prop111223b1}.
The implication~\eqref{prop111223b4}$\implies$\eqref{prop111223b1} is from Corollary~\ref{cor111225a}.

\eqref{prop111223b2}$\implies$\eqref{prop111223b4}
Assume that $\torfbm i{R/\m}C\cong\tormfc i{R/\m}C$ for all $i\geq 0$, for all $\m\in\mspec(R)$.
Corollary~\ref{cor111221b} then implies that
\begin{align*}
\torfbmm i{R_{\m}/\m_{\m}}{C_{\m}}
&\cong\torfbm i{R/\m}C_{\m}\\
&\cong\tormfc i{R/\m}C_{\m}\\
&\cong\tormfcm i{R_{\m}/\m_{\m}}{C_{\m}}.
\end{align*}
Because this is so for $i=0$ and some $i\geq 1$, Theorem~\ref{prop111223a} implies that $B_{\m}\cong R_{\m}\cong C_{\m}$.
This holds for all $\m$, so Fact~\ref{fact111224b} implies $B\approx R\approx C$.
\end{proof}

The next two results are proved similarly.

\begin{cor}\label{prop111223c'}
The following conditions are equivalent:
\begin{enumerate}[\rm(i)]
\item \label{prop111223c'1}
$\torfcm iXY\cong\torfcm iYX$ for all $i\geq 0$ and for all $R$-modules $X$, $Y$.
\item \label{prop111223c'1'}
$\tormfc iXY\cong\tormfc iYX$ for all $i\geq 0$ and for all $R$-modules $X$, $Y$.
\item \label{prop111223c'2}
$\torfcm iC{R/\m}\cong\torfcm i{R/\m}C$ for $i=0$, for some $i\geq 1$, and for all $\m\in\mspec(R)$.
\item \label{prop111223c'2'}
$\tormfc iC{R/\m}\cong\tormfc i{R/\m}C$ for $i=0$, for some $i\geq 1$, and for all $\m\in\mspec(R)$.
\item \label{prop111223c'3}
$C\approx R$.
\end{enumerate}
\end{cor}

\begin{cor}\label{prop111223d'}
The following conditions are equivalent:
\begin{enumerate}[\rm(i)]
\item \label{prop111223d'1}
$\torfcm iXY\cong\Tor iXY$ for all $i\geq 0$ and for all $R$-modules $X$, $Y$.
\item \label{prop111223d'1'}
$\tormfc iXY\cong\Tor iXY$ for all $i\geq 0$ and for all $R$-modules $X$, $Y$.
\item \label{prop111223d'2}
$\torfcm iC{R/\m}\cong\Tor iC{R/\m}$ for some $i\geq 1$, for all $\m\in\mspec(R)$.
\item \label{prop111223d'2'}
$\tormfc i{R/\m}C\cong\Tor i{R/\m}C$ for some $i\geq 1$, for all $\m\in\mspec(R)$.
\item \label{prop111223d'3}
$C\approx R$.
\item \label{prop111223d'4}
$\torfcm i{R/\m}{R/\m}\cong\Tor i{R/\m}{R/\m}$ for some $i\geq 0$, and for all $\m\in\mspec(R)$.
\item \label{prop111223d'4'}
$\tormfc i{R/\m}{R/\m}\cong\Tor i{R/\m}{R/\m}$ for some $i\geq 0$, and for all $\m\in\mspec(R)$.
\end{enumerate}
\end{cor}

\begin{disc}\label{disc111228a}
In spite of the general lack of balance properties for relative Tor, one still knows, for instance, that
$\ann_R(\torpcm iMN)\supseteq\ann_R(M)\cup\ann_R(N)$.
This follows, for instance, by Theorem~\ref{prop110730a'}
since $\ann_R(M)\subseteq\ann_R(\Hom CM)$ and $\ann_R(N)\subseteq\ann_R(\Otimes CN)$.
\end{disc}

\section{$\catfc$-Projective Dimension and Vanishing of Relative Homology}\label{sec111231d}

In this section,  $C$ is a semidualizing $R$-module, and  $M$ and $N$ are $R$-modules.

\

We begin this section with two results that are probably implicit in~\cite{takahashi:hasm}.
The first one is, in some sense, a counterpoint to Example~\ref{ex120116a}:
the example says that bounded and exact does not necessarily imply proper,
while the following lemma says that bounded and proper does imply exact.

\begin{lem}\label{lem120121c}
Assume  that $\fcpd_R(M)\leq n$ and let $L$ be a proper $\catfc$-resolution of $M$ such that $L_i=0$ for $i>n$.
Then $L^+$ is exact and we have $M\in\catbc(R)$.
\end{lem}

\begin{proof}
Lemma~\ref{lem111220a}\eqref{lem111220a2} implies that 
the complex $\Hom CL$ is  a proper flat resolution
of $\Hom CM$ such that  $\Hom CL_i=0$ for $i>n$. In particular, we have $\fd_R(\Hom CM)\leq n$,
so $\Hom CM\in\catac(R)$. By Foxby equivalence, we conclude that $M\in\catbc(R)$, so $M\cong\Otimes C{\Hom CM}$;
see Remark~\ref{disc111220a}(a).
The conditions $L\cong\Otimes C{\Hom CL}$ and $M\cong\Otimes C{\Hom CM}$ imply that $L^+\cong\Otimes C\Hom CL^+$, so 
Lemma~\ref{lem110730a}\eqref{lem110730a1} implies that $L^+$ is exact.
Since each $L_i$ is in $\catbc(R)$, the condition $M\in\catbc(R)$ follows from the two-of-three property
in Remark~\ref{disc111220a}.
\end{proof}

\begin{prop}\label{lem111230a}
\begin{enumerate}[\rm(a)]
\item \label{lem111230a1}
One has $\fcpd_R(M)\leq n$ if and only if there is an exact sequence
$0\to L_n\to\cdots\to L_0\to M\to 0$ such that each $L_i\in\catfc(R)$.
\item \label{lem111230a2}
$\fcpd_R(M)=\fd_R(\Hom CM)$.
\item \label{lem111230a3}
$\fcpd_R(\Otimes CM)=\fd_R(M)$.
\end{enumerate}
\end{prop}

\begin{proof}
\eqref{lem111230a1}
Assume first that $\fcpd_R(M)\leq n$ and let $L$ be a proper $\catfc$-resolution of $M$ such that $L_i=0$ for $i>n$.
Lemma~\ref{lem120121c} implies that 
$L^+$
is an exact sequence of the desired form.

Conversely, assume that there is an exact sequence
$$L^+=(0\to L_n\to\cdots\to L_0\to M\to 0)$$ 
such that each $L_i\in\catfc(R)$.
The two-of-three property for Bass classes implies that $M\in\catbc(R)$; see Remark~\ref{disc111220a}.
Lemma~\ref{lem110730aa}\eqref{lem110730aa1} implies that 
$\Hom C{L^+}$ is exact, that is, it is an augmented flat resolution of $\Hom CM$ such that
$\Hom CL_i=0$ for $i>n$. So we have $\fd_R(\Hom CM)\leq n$.
Lemma~\ref{lem111231a} implies that a truncation
$$T^+=(0\to K_n\to\Hom C{L_{n-1}}\to\cdots\to \Hom C{L_{0}}\to\Hom CM\to 0)$$
is an exact proper flat resolution of $\Hom CM$.
From Lemma~\ref{lem111220a}\eqref{lem111220a1} we conclude that
$U=\Otimes CT$ is a proper $\catfc$-resolution of $\Otimes C\Hom CM\cong M$
such that $U_i=0$ for all $i\geq n$, so $\fcpd_R(M)\leq n$ as desired.

\eqref{lem111230a2} 
The proof of Lemma~\ref{lem120121c}
implies that $\fcpd_R(M)\geq\fd_R(\Hom CM)$.
For the reverse inequality, assume that $\fd_R(\Hom CM)=m<\infty$.
Given a proper flat resolution $F$ of $\Hom CM$, Lemma~\ref{lem111231a} implies that $F$ has a truncation
$F'$ that is a proper flat resolution of $\Hom CM$ such that $F'_i=0$ for all $i>m$.
As in the previous paragraph, it follows that $\fcpd_R(M)\leq m=\fd_R(\Hom CM)$.

\eqref{lem111230a3}
To show that $\fcpd_R(\Otimes CM)\geq\fd_R(M)$, assume without loss of generality that 
$\fcpd_R(\Otimes CM)<\infty$. It follows that $\Otimes CM\in\catbc(R)$, so $M\in\catac(R)$ by Foxby equivalence;
see Remark~\ref{disc111220a}(b).
Part~\eqref{lem111230a2} implies that 
$$\fcpd_R(\Otimes CM)=\fd_R(\Hom C{\Otimes CM})=\fd_R(M)$$
as desired. The reverse inequality is verified similarly.
\end{proof}

\begin{cor}\label{cor111231b}
Let $R\to S$ be a flat ring homomorphism.
Then there is an inequality
$\fcpd_R(M)\geq\catf_{\Otimes SC}\text{-}\pd_S(\Otimes SM)$
with equality holding when $S$ is faithfully flat over $R$.
\end{cor}

\begin{proof}
Given an $R$-module $N$, it is routine to show that
$\fd_R(N)\geq\fd_S(\Otimes SN)$ with equality holding when $S$ if faithfully flat over $R$.
This explains the second step in the next display
\begin{align*}
\fcpd_R(M)
&=\fd_R(\Hom CM)\\
&\geq\fd_S(\Otimes S{\Hom CM})\\
&=\fd_S(\Hom[S]{\Otimes SC}{\Otimes SM})\\
&=\catf_{\Otimes SC}\text{-}\pd_S(\Otimes SM).
\end{align*}
The first and fourth steps are from Proposition~\ref{lem111230a}\eqref{lem111230a2},
and the third step is by the isomorphism
$\Otimes S{\Hom CM}\cong\Hom[S]{\Otimes SC}{\Otimes SM}$.
When $S$ is faithfully flat over $R$, we have equality in the second step, hence the desired conclusions.
\end{proof}

\begin{cor}\label{cor111231a}
Given  an integer $n\geq 0$,
the following conditions are equivalent:
\begin{enumerate}[\rm(i)]
\item \label{cor111231a1}
$\fcpd_R(M)\leq n$.
\item \label{cor111231a2}
$\catf_{U^{-1}C}\text{-}\pd_{U^{-1}R}(U^{-1}M)\leq n$ for each multiplicatively closed subset $U\subseteq R$.
\item \label{cor111231a3}
$\catf_{C_{\p}}\text{-}\pd_{R_{\p}}(M_{\p})\leq n$ for each $\p\in\spec(R)$.
\item \label{cor111231a4}
$\catf_{C_{\m}}\text{-}\pd_{R_{\m}}(M_{\m})\leq n$ for each $\m\in\mspec(R)$.
\end{enumerate}
\end{cor}

\begin{proof}
This follows from Proposition~\ref{lem111230a}
like Corollary~\ref{cor111231b}, using the local global principal for flat dimension.
\end{proof}

\begin{fact}\label{fact110811a}
Let $E$ be an injective $R$-module.
The  next facts are essentially contained in~\cite[Lemmas 4.1 and 4.2]{sather:abc}. See also Fact~\ref{fact111230a}
and Proposition~\ref{lem111230a}.
\begin{enumerate}[(a)]
\item\label{fact110811a1}
If $N$ is $C$-injective, then $\Hom NE$ is $C$-flat.
As a consequence, we have $\fcpd_R(\Hom NE)\leq\icid_R(N)$.
In particular, if $\icid_R(N)<\infty$, then $\fcpd_R(\Hom NE)<\infty$.
When $E$ is faithfully injective, the converses of the first and third statements hold, and equality holds in the second statement.
\item\label{fact110811a2}
If $N$ is $C$-flat, than $\Hom NE$ is $C$-injective.
As a consequence, we have $\icid_R(\Hom NE)\leq\fcpd_R(N)$.
Hence, if $\fcpd_R(\Hom NE)<\infty$, then $\icid_R(N)<\infty$.
When $E$ is faithfully injective, the converses of the first and third statements hold, and equality holds in the second statement.
\end{enumerate}
\end{fact}

The next two results contain Theorem~\ref{intthm120115b} from the introduction.

\begin{thm}\label{prop110811d}
Given an integer $n\geq 0$, the following conditions are equivalent:
\begin{enumerate}[\rm(i)]
\item \label{prop110811d1}
$\torfcm{i} M- = 0$ for all $i>n$;
\item \label{prop110811d2}
$\torfcm{n+1} M- = 0$; and
\item \label{prop110811d3}
$\fcpd_R(M)\leq n$.
\end{enumerate}
\end{thm}

\begin{proof}
Let $E$ be a faithfully injective $R$-module, and set $(-)^\vee=\Hom -E$.
Condition~\eqref{prop110811d1} is equivalent to the following, since $E$ is faithfully injective:
\begin{enumerate}[\rm(i)]
\item[($\text{i}'$)] $\torfcm iM-^\vee=0$ for all $i>n$.
\end{enumerate}
Since $\torfcm iM-\cong\tormfc i-M$, 
Theorem~\ref{prop110730a'} and Proposition~\ref{thm111230a} imply that ($\text{i}'$) is equivalent to the following:
\begin{enumerate}[\rm(i)]
\item[($\text{i}''$)] $\extmic{i} -{M^\vee}=0$ for all $i>n$.
\end{enumerate}
Similarly, condition~\eqref{prop110811d2} is equivalent to the following:
\begin{enumerate}[\rm(i)]
\item[($\text{ii}''$)] $\extmic{n+1} -{M^\vee}=0$.
\end{enumerate}
Condition~\eqref{prop110811d3} is equivalent to the following, by Fact~\ref{fact110811a}\eqref{fact110811a2}:
\begin{enumerate}[\rm(i)]
\item[($\text{iii}''$)] $\icid_R(M^\vee)\leq n$.
\end{enumerate}
Fact~\ref{fact110811d} shows that the conditions ($\text{i}''$)--($\text{iii}''$) are equivalent.
Thus, the conditions~\eqref{prop110811d1}--\eqref{prop110811d3} from the statement of the theorem are equivalent.
\end{proof}

\begin{thm}\label{prop111231a}
Assume that $M$ is finitely generated over $R$. Given an integer $n\geq 0$, the following conditions are equivalent:
\begin{enumerate}[\rm(i)]
\item \label{prop111231a1}
$\torfcm{i} M{R/\m} = 0$ for all $i>n$ and for each $\m\in\mspec(R)$;
\item \label{prop111231a2}
$\torfcm{n+1} M{R/\m} = 0$ for each $\m\in\mspec(R)$; 
\item \label{prop111231a4}
$\pcpd_R(M)\leq n$; and
\item \label{prop111231a3}
$\fcpd_R(M)\leq n$.
\end{enumerate}
\end{thm}

\begin{proof}
The implication~\eqref{prop111231a1}$\implies$\eqref{prop111231a2} is trivial,
and~\eqref{prop111231a3}$\implies$\eqref{prop111231a1} is from Theorem~\ref{prop110811d}.

\eqref{prop111231a2}$\implies$\eqref{prop111231a4}
Assume that $\torfcm{n+1} M{R/\m} = 0$ for each $\m\in\mspec(R)$.
The module $C_{\m}$ is a semidualizing $R_{\m}$-module,
so it is non-zero and finitely generated.
Thus, we have
\begin{equation}\label{eq111231a}
\Otimes C{R/\m}\cong C/\m C\cong C_{\m}/\m_{\m}C_{\m}\cong (R/\m)^{\beta_0(\m;C)}
\end{equation}
where $\beta_0(\m;C)\neq 0$.

The second step in the next sequence is by Theorem~\ref{prop110730a'}:
\begin{align*}
0
&=\torfcm{n+1} M{R/\m} \\
&\cong\Tor{n+1}{\Hom CM}{(\Otimes C{R/\m})}\\
&\cong\Tor{n+1}{\Hom CM}{R/\m}^{\beta_0(\m;C)}.
\end{align*}
The third step is by~\eqref{eq111231a}. Since $\beta_0(\m;C)\neq 0$, we conclude that
$$\Tor{n+1}{\Hom CM}{R/\m}=0$$ 
for each $\m$. Thus, Proposition~\ref{lem111230a}\eqref{lem111230a2} explains the first step in the next display
$$\fcpd_R(M)=\fd_R(\Hom CM)=\pd_R(\Hom CM)\leq n$$
and the remaining steps follow from the fact that $\Hom CM$ is finitely generated.

\eqref{prop111231a4}$\implies$\eqref{prop111231a3}
Assume that $\pcpd_R(M)\leq n$.
Then Fact~\ref{fact111230a}\eqref{fact111230a1} provides an exact sequence
$$0\to L_n\to\cdots\to L_0\to M\to 0$$ such that each $L_i\in\catpc(R)$.
In particular, we have  $L_i\in\catfc(R)$, so Proposition~\ref{lem111230a}\eqref{lem111230a1}
implies that $\fcpd_R(M)\leq n$.
\end{proof} 

\begin{cor}\label{cor120115a}
If $M$ is finitely generated, then
$\fcpd_R(M)=\pcpd_R(M)$.
\end{cor}

\begin{cor}\label{prop111231b}
Given  a set $\{N_j\}_{j\in J}$ of $R$-modules,
one has
$$\textstyle
\fcpd_R(\coprod_jN_j)=\sup\{\fcpd_R(N_j)\mid j\in J\}.$$
\end{cor}

\begin{proof}
Apply Theorem~\ref{prop110730a'},
Proposition~\ref{prop111228a}\eqref{prop111228a1},
and Theorem~\ref{prop110811d}.
\end{proof}

We conclude this section with a two-of-three result for modules of finite $\catfc$-projective dimension.

\begin{cor}\label{cor120126a}
Given an exact sequence $\mathbb{M}=(0\to M'\to M\to M''\to 0)$
of $R$-module homomorphisms, one has
\begin{align*}
\fcpd_R(M)&\leq\sup\{\fcpd_R(M'),\fcpd_R(M'')\} \\
\fcpd_R(M')&\leq\sup\{\fcpd_R(M),\fcpd_R(M'')-1\} \\
\fcpd_R(M'')&\leq\sup\{\fcpd_R(M),\fcpd_R(M')+1\} . 
\end{align*}
In particular, if two of the modules in $\mathbb M$ have finite $\catfc$-projective dimension, then so does the third module.
\end{cor}

\begin{proof}
For each inequality, one can assume without loss of generality that two of the modules in the sequence
have finite $\catfc$-projective dimension. In particular, these modules are in $\catbc(R)$, so the two-of-three
property implies that all three modules are in $\catbc(R)$; see Remark~\ref{disc111220a}. In particular, we have
$\Ext 1C{M'}=0$, so the sequence $\Hom C{\mathbb M}$ is exact. Lemma~\ref{lem111224a} implies that
$\mathbb M$ is $\Hom {\catpc}-$-exact,  so
Proposition~\ref{prop111220b}\eqref{prop111220b1}
and Theorem~\ref{prop110730a'}
provide a long exact sequence 
$$\cdots\torfcm {i+1}{M''\!}{N}\to\torfcm i{M'}{N}\to\torfcm i{M}{N}\to\torfcm i{M''\!}{N}\cdots$$
for each $R$-module $N$.
The desired inequalities follow by analyzing the vanishing in this sequence using
Theorem~\ref{prop110811d}. 
\end{proof}

\section{Pure Submodules}\label{sec110818a}

In this section,  $C$ is a semidualizing $R$-module, and  $M$ is an $R$-module.

\begin{defn}
An $R$-submodule $M'\subseteq M$ is \emph{pure} if for every $R$-module $N$ the induced map
$\Otimes N{M'}\to\Otimes N{M}$ is injective.
An exact sequence
$$\mathbb M=(0\to M'\to M\to M''\to 0)$$
is \emph{pure} if for every $R$-module $N$ the
sequence $\Otimes N{\mathbb M}$ is exact.
\end{defn}

\begin{disc}\label{disc110818a}
An $R$-submodule $M'\subseteq M$ is pure if and only if the induced sequence
$0\to M'\to M\to M/M'\to 0$ is pure.
\end{disc}

The next fact is from~\cite[Proposition 3]{warfield:pacm}.

\begin{fact}\label{fact110818a}
Let $M'\subseteq M$ be an $R$-submodule, and consider the natural
exact sequence $\mathbb M=(0\to M'\to M\to M''\to 0)$. The following conditions are equivalent:
\begin{enumerate}[(i)]
\item\label{fact110818a1}
$M'$ is a pure submodule of $M$.
\item\label{fact110818a2}
for each finitely presented (i.e., finitely generated) $R$-module $N$, the induced map
$\Hom{N}{M}\to\Hom{N}{M''}$ is surjective.
\item\label{fact110818a3}
for each finitely presented $R$-module $N$, the sequence $\Hom{N}{\mathbb M}$ is exact.
\end{enumerate}
\end{fact}

This fact yields our next result which applies, e.g., when $L$ is semidualizing. 

\begin{prop}\label{prop110818a}
Let $M'\subseteq M$ be a pure submodule, and let $L$ be a finitely generated $R$-module.
Then the submodule $\Hom{L}{M'}\subseteq\Hom{L}{M}$ is pure.
\end{prop}

\begin{proof}
Set $\mathbb M=(0\to M'\to M\to M''\to 0)$, and let $N$ be an $R$-module.
Note that if $N$ and $L$ are finitely generated, then so is $\Otimes NL$.
Now use Fact~\ref{fact110818a} with  Hom-tensor adjointness:
$\Hom{N}{\Hom{L}{\mathbb M}}\cong\Hom{\Otimes{N}{L}}{\mathbb M}$.
\end{proof}

The next result generalizes~\cite[Lemma 9.1.4]{enochs:rha} and~\cite[Lemma 5.2(a)]{holm:fear} in our setting.
It is Theorem~\ref{xthm111231a} from the introduction.

\begin{thm}\label{thm111231a}
Let
$M'\subseteq M$ be a pure submodule. Then one has
$$\fcpd_R(M)\geq\sup\{\fcpd_R(M'),\fcpd_R(M/M')-1\}.$$
\end{thm}

\begin{proof}
Assume without loss of generality that
$\fcpd _{R}(M)=n < \infty$. It follows that $M\in\catbc(R)$, and from~\cite[Proposition 2.4(a) and Theorem 3.1]{holm:cpdp}
we know that $M'$ and $M'':=M/M'$ are in $\catbc(R)$.
In particular, the sequence 
\begin{equation}
0\to M'\to M\to M''\to 0\label{eq111231b}
\end{equation}
is $\Hom C-$-exact, so it is $\Hom{\catpc}{-}$-exact by Lemma~\ref{lem111224a}.

We  prove that $\fcpd_{R} (M') \leq \fcpd _{R}(M)$.
By Theorems~\ref{prop110730a'} and~\ref{prop110811d}, we have
$$\Tor {n+1} {\Hom {C}{M}} {\Otimes C-}\cong\torfcm {n+1} {M}{-}= 0.$$
Let $G$
 be an arbitrary $R$-module and let $$0 \to K_{n+1} \to P_{n} \to \cdots \to P_{0}
 \to C\otimes _{R} G \to 0,$$ be a truncation of a projective resolution of $\Otimes CG$.
In the commutative diagram
$$\xymatrix{
0\ar[r]&\Otimes {K_{n+1}}{\Hom C{M'}} \ar[r]\ar[d]&  \Otimes {P_{n}}{\Hom C{M'}} \ar[d] \\
0 \ar[r]&  \Otimes {K_{n+1}}{\Hom {C}{M}} \ar[r] &  \Otimes {P_{n}}{\Hom {C}{M}}
}$$
the bottom row is exact, since $\Tor {n+1} {\Hom CM}  {\Otimes {C}{G}} = 0$.
The two vertical arrows are injective, since $\Hom C{M'}  \subseteq \Hom CM$ is
pure by Proposition~\ref{prop110818a}. Hence, the 
top row of the diagram 
is exact, so we have
$$\torfcm {n +1} {M'}G\cong\Tor {n +1} {\Hom C{M'}}{\Otimes CG} = 0$$
by Theorem~\ref{prop110730a'}. Since $G$ was chosen arbitrarily, we conclude that
$\fcpd _{R}(M') \leq n=\fcpd _{R}(M)$, by Theorem~\ref{prop110811d}.

To complete the proof, we need only observe that Corollary~\ref{cor120126a} implies that  $\fcpd_R(M'')-1\leq n$.
\end{proof}

\begin{ex}\label{ex111231a}
Let $M'$ and $M''$ be $R$-modules such that
$$\fcpd_R(M')<\fcpd_R(M'')<\infty.$$
The trivial exact sequence
$0\to M'\to M'\oplus M''\to M''\to 0$ is split
hence pure, so
\begin{align*}
\fcpd_R(M'\oplus M'')=\fcpd_R(M'')>\sup\{\fcpd_R(M'),\fcpd_R(M'')-1\}
\end{align*}
by Proposition~\ref{prop111231b}. Thus, we can have strict inequality in Theorem~\ref{thm111231a}.
\end{ex}


\providecommand{\bysame}{\leavevmode\hbox to3em{\hrulefill}\thinspace}
\providecommand{\MR}{\relax\ifhmode\unskip\space\fi MR }
\providecommand{\MRhref}[2]{%
  \href{http://www.ams.org/mathscinet-getitem?mr=#1}{#2}
}
\providecommand{\href}[2]{#2}

\end{document}